\documentclass[a4paper,11pt]{article}

\usepackage[titletoc,toc,title]{appendix}

\usepackage{amsfonts}
\usepackage{latexsym}
\usepackage{amssymb}
\usepackage{amsmath}
\usepackage{times}
\usepackage{amsthm}
\usepackage{mathrsfs}
\usepackage{graphics}

\newtheorem*{theoremx}{Theorem}

\newtheorem{theorem}{Theorem}[section]
\newtheorem{proposition}[theorem]{Proposition}

\newtheorem{lemma}[theorem]{Lemma}

\theoremstyle{definition}
\newtheorem{definition}[theorem]{Definition}
\newtheorem{example}[theorem]{Example}

\theoremstyle{remark}
\newtheorem{remark}[theorem]{Remark}

\newcommand{\eps}{\varepsilon}
\newcommand{\ovl}{\overline}

\newcommand{\R}{\mathbb{R}}

\newcommand{\C}{\mathbb{C}}

\newcommand{\G}{\mathcal{G}}

\newcommand{\F}{\mathcal{F}}

\newcommand{\Z}{\mathbb{Z}}
\newcommand{\N}{\mathbb{N}}

\renewcommand{\P}{\mathbb{P}}

\newcommand{\U}{\mathscr{U}}

\newcommand{\Pc}{\mathcal{P}}

\newcommand{\ds}{\displaystyle}

\DeclareMathOperator{\sing}{Sing}
  \DeclareMathOperator{\gen}{gen}

 \DeclareMathOperator{\tang}{Tang}

\title{A new proof of the classification of Elliptic Foliations induced by real Quadratic Fields with center}

\author{
Liliana Puchuri\thanks{Instituto de Mat\'ematica and Ciencias Afines (IMCA) \& Pontificia Universidad Cat\'olica del Per\'u. Email: {\tt lpuchuri@pucp.edu.pe}} 
\and 
Orestes Bueno
\thanks{Universidad del Pac\'ifico \& Instituto de Matem\'atica and Ciencias Afines (IMCA). Email: {\tt o.buenotangoa@up.edu.pe}}}

\date{}

\begin{document}

\maketitle

\begin{abstract}
In this work, we give a new proof of the classification of the Lotka-Volterra and Reversible foliations, originally given by Gautier. This new proof, involves an unified technique for both cases, using the theory of foliations. 
In addition, we obtain a linear family of elliptical foliations with a non-invariant tangency set.
%

%




\bigskip


\bigskip

\end{abstract}

\section{Introduction}
The \emph{infinitesimal's Hilbert Problem} asks for an upper bound to the number of limit cycles of a polinomial vector field of degree $n$, close to a polinomial vector field with first integral $f$.
Even the case $n=2$ is an open problem. In this case, there is some progress when $f$ has elliptic curves as generic level curves (called \emph{elliptic fibrations})~\cite{Gau1,Gavri1,Ilya2002,Petrov90}.

%

Any quadratic differential equation, for which the origin is a non-degenerated singularity of center type, can be taken to the following form
\begin{equation*}\label{eqcua}
\begin{aligned}
x'&=y+a_{2,0}x^2+a_{1,1}xy+a_{0,2}y^2,\\
y'&=-x+b_{2,0}x^2+b_{1,1}xy+b_{0,2}y^2.
\end{aligned}
\end{equation*}
We can also complexify the previous equation,
to obtain
\begin{equation}\label{eqcent2}
z'=-i z+Az^2+Bz\bar{z}+C\bar{z}^2,
\end{equation}
where $A,B,C\in\C$. 

The integrability theory of Darboux~\cite{Darboux1878} made it possible to obtain necessary and sufficient conditions for the classification theorem of centers of quadratic polynomial differential systems. This was achieved primarily by Kapteyn~\cite{Kap,Kap2} and Bautin~\cite{Bautin}. 
%
\begin{theoremx}[Kapteyn-Bautin]\label{teo:kapteyn}
There are five types of quadratic systems with center:
\begin{description}
\item[$H$:]$ z'=-i z+-z^2+2z\bar{z}+C\bar{z}^2$, $C\in \C \setminus \R$, (Hamiltonian);
\item[$H_1$:] $ z'=-i z+\bar{z}^2$, (Hamiltonian 1);
\item[$Q_3^L$:] $ z'=-i z+z^2+C\bar{z}^2$,  $C\in \C$, (generalized Lotka-Volterra);
\item[$Q_3^{R}$:] $z'=-i a z+4 z^2+2z\bar{z}+c\bar{z}^2,  \ a, c \in \R$, (Reversible);
\item[$Q_4$:]$z'=-i z+4z^2+2z\bar{z}+C\bar{z}^2$,  $|C|=2, C \in \C\setminus \R$, (Codimension 4). 
\end{description}
\end{theoremx}

From Kapteyn-Bautin's theorem, we obtain the classification of quadratic vector fields with a center, namely, if the complex ODE~\eqref{eqcent2} possesses a center then it must have a first integral of one of the following forms
\begin{gather}
 P_3 \in \R[x,y],\quad\text{(Hamiltonian cases: $H$ and $H_1$);}\nonumber\\
 x^py^q(ax+by+c)^r,\quad p,q \in \Z,\, a,b,c\in \R,\quad\text{(Lotka-Volterra case: $Q_3^L$);}\label{eq:lotka}\\
x^p(y^2+P_2(x))^q,\quad q \in \N,\,  p\in \Z,\,
P_2 \in \R_2[x,y],\quad\text{(Reversible case: $Q_3^R$);}\label{eq:rever}\\
\dfrac{P_3(x,y)^2}{P_2(x,y)^3},\quad P_2,P_3 \in \R[x,y],\quad \text{(Codimensi\'on 4 case: $Q_4$).}\nonumber 
\end{gather} 
In~\cite{Gau1}, Gautier provides the classification of reversible and Lotka-Volterra foliations. For this, Gautier uses two vastly different approaches.  For reversible foliations he uses the genus formula for hyperelliptic curves, whereas for Lotka-Volterra foliations, he calculates the number of zeros and poles of a certain 1-form to obtain the genus of the generic fiber. 

In this work, we give a different proof of the classification of the Lotka-Volterra and Reversible folations. 
For the Reversible case, we recover the following theorem (see Section~\ref{sec:reversible}, Theorem~\ref{teo:reversible}).
\begin{theoremx} 
Let $f$ be defined as
\[
f(x,y,z)=\dfrac{x^{p}(y^2+ax^2+bxz+cz^2)^q}{z^{p+2q}}, 
\]
and let $\F$ be the foliation induced by $df$. Then $\F$ is elliptic if, and only if, after an automorphism of $\P^2$, it has a first integral of the form:
\begin{enumerate}
 \item If $p+2q>0$, $a\neq 0$ and $c\neq 0$:
 \begin{gather*}
f(x,y,z)=\dfrac{(y^2+ax^2+bxz+cz^2)^2}{xz^{3}},\quad f(x,y,z)=\dfrac{(y^2+ax^2+bxz+cz^2)^2}{x^3z}\\
\end{gather*}
\item If $p+2q>0$, $ab\neq 0$ and $c=0$:
\begin{gather*}
f(x,y,z)=\dfrac{(y^2+ax^2+bxz)^2}{xz^{3}},\quad f(x,y,z)=\dfrac{(y^2+ax^2+bxz)^3}{x^2z^{4}},\\
f(x,y,z)=\dfrac{(y^2+ax^2+bxz)^3}{x^{4}z^{2}},\quad f(x,y,z)=\dfrac{(y^2+ax^2+bxz)^3}{x^{5}z},\\
f(x,y,z)=\dfrac{(y^2+ax^2+bxz)^4}{x^{5}z^{3}},\quad f(x,y,z)=\dfrac{(y^2+ax^2+bxz)^4}{x^{7}z},
\end{gather*}
\item If $p+2q>0$, $a=c=0$:
 \begin{gather*}
f(x,y,z)=\dfrac{(y^2+bxz)^{1+6u}}{x^{-2+6u}z^{4+6u}},\quad f(x,y,z)=\dfrac{(y^2+bxz)^{-1+6u}}{x^{-4+6u}z^{2+6u}},\\
f(x,y,z)=\dfrac{(y^2+bxz)^{2+6u}}{x^{-1+6u}z^{5+6u}},\quad f(x,y,z)=\dfrac{(y^2+bxz)^{-2+6u}}{x^{-5+6u}z^{1+6u}},\\
f(x,y,z)=\dfrac{(y^2+bxz)^{1+6u}}{x^{4+6u}z^{-2+6u}},\quad f(x,y,z)=\dfrac{(y^2+bxz)^{-1+6u}}{x^{3+6u}z^{-5+6u}},\\
f(x,y,z)=\dfrac{(y^2+bxz)^{2+6u}}{x^{5+6u}z^{-1+6u}},\quad f(x,y,z)=\dfrac{(y^2+bxz)^{-2+6u}}{x^{1+6u}z^{-5+6u}},\\
f(x,y,z)=\dfrac{(y^2+bxz)^{1+2u}}{x^{-1+2u}z^{3+2u}},\quad f(x,y,z)=\dfrac{(y^2+bxz)^{1+2u}}{x^{3+2u}z^{-1+2u}},
\end{gather*}
for any $u\in\N$,
\item If $p+2q<0$ and $c\neq 0$:
 \begin{gather*}
f(x,y,z)=\dfrac{(y^2+ax^2+bxz+cz^2)z^2}{x^{4}},\quad f(x,y,z)=\dfrac{(y^2+ax^2+bxz+cz^2)z}{x^{3}},\\
\end{gather*}
\item If $p+2q<0$ and $c=0$:
 \begin{gather*}
f(x,y,z)=\dfrac{(y^2+ax^2+bxz)z^2}{x^{4}},\quad f(x,y,z)=\dfrac{(y^2+ax^2+bxz+cz^2)z}{x^{3}},\\
f(x,y,z)=\dfrac{(y^2+ax^2+bxz)^2z}{x^{5}}.\\
\end{gather*}
\end{enumerate}
\end{theoremx}

On the other side, for the Lotka-Volterra case, we have the following theorem (see Section~\ref{sec:lotkavolterra}, Theorem~\ref{teo:lotkavolterra}). 
\begin{theoremx}
Let $f$ be defined as 
\[
f(x,y,z)=\dfrac{x^{p}y^{q}(ax+by+cz)^r}{z^{p+q+r}}, 
\]
and let $\F$ be the foliation induced by $df$. Then $\F$ is elliptic if, and only if, after an automorphism of $\P^2$, it has a first integral of the form:
\begin{enumerate}\renewcommand{\theenumi}{\Roman{enumi}}
 \item $ab\neq 0$, $c\in\R$ and $p>0$, $q>0$:
\begin{gather*}
f(x,y,z)=\dfrac{xy(ax+by+cz)}{z^3}, \quad f(x,y,z)=\dfrac{xy(ax+by+cz)^2}{z^4}, \\
f(x,y,z)=\dfrac{xy^2(ax+by+cz)^3}{z^6},
\end{gather*}
\item $abc\neq 0$ and $p<0$, $q>0$, $p+q+r>0$:
\begin{gather*}
^{\dagger}f(x,y,z)=\dfrac{y(ax+by+cz)^3}{x^2z^2},\quad {f(x,y,z)=\dfrac{y^2(ax+by+cz)^2}{xz^{3}}},\\
^{\dagger}f(x,y,z)=\dfrac{y(ax+by+cz)^4}{x^2z^{3}},\quad ^{\dagger}f(x,y,z)=\dfrac{y^2(ax+by+cz)^3}{xz^{4}},\\
^{\dagger}f(x,y,z)=\dfrac{y(ax+by+cz)^6}{x^3z^{4}},\quad ^{\dagger}f(x,y,z)=\dfrac{y^3(ax+by+cz)^4}{xz^{6}},
\end{gather*}
\item $a=0$, $bc\neq 0$ and $p>0$, $q>0$:
\begin{gather*}
f(x,y,z)=\dfrac{x^{3}y^{1+3u}(by+cz)^{1+3v}}{z^{5+3(u+v)}},\quad f(x,y,z)=\dfrac{x^{3}y^{2+3u}(by+cz)^{2+3v}}{z^{7+3(u+v)}},\\
f(x,y,z)=\dfrac{x^{4}y^{1+4u}(by+cz)^{1+4v}}{z^{6+4(u+v)}},\quad f(x,y,z)=\dfrac{x^{4}y^{3+4u}(by+cz)^{3+4v}}{z^{10+4(u+v)}},\\
f(x,y,z)=\dfrac{x^{4}y^{1+4u}(by+cz)^{2+4v}}{z^{7+4(u+v)}},\quad f(x,y,z)=\dfrac{x^{4}y^{2+4u}(by+cz)^{3+4v}}{z^{9+4(u+v)}},\\
f(x,y,z)=\dfrac{x^{6}y^{1+6u}(by+cz)^{2+6v}}{z^{9+6(u+v)}},\quad f(x,y,z)=\dfrac{x^{6}y^{4+6u}(by+cz)^{5+6v}}{z^{15+6(u+v)}},\\
f(x,y,z)=\dfrac{x^{6}y^{1+6u}(by+cz)^{3+6v}}{z^{10+6(u+v)}},\quad f(x,y,z)=\dfrac{x^{6}y^{3+6u}(by+cz)^{5+6v}}{z^{14+6(u+v)}},\\
f(x,y,z)=\dfrac{x^{6}y^{2+6u}(by+cz)^{3+6v}}{z^{11+6(u+v)}},\quad f(x,y,z)=\dfrac{x^{6}y^{3+6u}(by+cz)^{4+6v}}{z^{13+6(u+v)}},
\end{gather*}
for every $u,v\geq 0$, integers.
\end{enumerate}
In addition, for every first integral of the form
\[
f(x,y,z)=\dfrac{x^py^q(by+cz)^r}{z^{p+q+r}}
\]
in case III, with $p<q+r$, we must also consider a first integral of the form
\[
f(x,y,z)=\dfrac{x^{p-q-r}y^q(ax+by)^r}{z^{p}}=\dfrac{y^q(ax+by)^r}{x^{-p'}z^{p'+q+r}},\qquad p'=p-q-r.
\]
\end{theoremx}

The main tool of our new proofs is Theorem~\ref{cerveaulins} due to Cerveau and Lins-Neto~\cite{CeLi}. 
We use such theorem to calculate the genus of the generic fiber of a first integral of the foliation. Thus we device an unified technique which involves the theory of foliations.  

We address the classification of reversible foliations in Section~\ref{sec:reversible} and the classification of Lotka-Volterra foliations in Section~\ref{sec:lotkavolterra}.  We must remark that, in the latter case, we obtain additional foliations apart from the originally obtained by Gautier, namely, every foliation induced by the first integrals marked with $\dagger$ in the above theorem.

In Section~\ref{sec:Linear}, we deal with pencils of foliations (see~\cite{LN3}). Theorem~\ref{teo:lins7}~\cite{LN3} give us a classification of four pencils of elliptic foliations whose tangency set is invariant. In addition, Proposition~\ref{lem4cap2-2} provide a characterizations of such foliations. The linear families obtained from Gautier's classification allow us to find many examples of pencils formed by foliations induced by elliptic fibrations and whose tangency set is non-invariant.
\section{Preliminaries}\label{sec:prelim}

 
\begin{definition}
An \emph{automorphism} of $\P^2$ is a map $F:\P^2\to\P^2$,
\[
F[x:y:z]=[a_{11}x+a_{12}y+a_{13}z:a_{21}x+a_{22}y+a_{23}z:a_{31}x+a_{32}y+a_{33}z],  
\]
where the matrix $A=[a_{ij}]$ is non-singular.
\end{definition}

\begin{definition}
Let $\F$ be a foliation on $\P^2$. We say that $\F$ is \emph{reversible} (respectively, \emph{Lotka-Volterra}), if, after an automorphism on $\P^2$, it possesses a first integral of the form~\eqref{eq:rever}, but not~\eqref{eq:lotka} (respectively, of the form~\eqref{eq:lotka}, but not~\eqref{eq:rever}).
\end{definition}

We will recall some definitions about fibrations in complex compact surfaces. Let $X$ be a compact surface and let $S$ be a compact Riemann surface. A \emph{fibration} is an holomorphic map $f:X\to S$.  A fibration is called \emph{rational} (respectively, \emph{elliptic}), if all but finitely many fibers have genus zero (respectively, genus one).

Let $\F$ be a foliation in $\P^2$ and let $\pi:\widetilde{\P^2}\to \P^2$ be the desingularization of $\F$. We say that $\F$ is \emph{elliptic} if it possesses a first integral $F:\P^2  \dashrightarrow \P^1$ such that $F\circ\pi:\widetilde{\P^2}\to \P^1$ is an elliptic fibration.

Our technique involves a way of calculate the genus of an irredutible curve, invariant by a foliation, using a certain  \emph{multiplicity} of the asociated field.

\begin{definition}\label{invcurva}
Let $U\subset\C^2$, $V\subset\C$ be open sets, with $0\in V$. Let $X$ be a vector field on $U$ and $f:U\to V$ holomorphic. We say that $C=f^{-1}(0)$ is \emph{invariant by $X$} if
\[
df_q(X(q))=0,\quad\forall q\in C.
\]
\end{definition}

%
%


\begin{proposition}[{\cite[Proposition 3]{CSL}}]\label{propmulti}
Let $X$ be a field on the open set $U\subset\C^2$, $S$ a one dimensional invariant submanifold, and $p\in S$ an isolated singularity of $X$. Let $\alpha:V\to U$ the Puiseux parametrization on a domain $V\subset C$, which contains $p$.
Then, there exists a unique holomorphic vector field $X_1$ on $D$ such that
 %
\[
d\alpha \cdot X_1=X\circ \alpha.
\]
\end{proposition}

In the previous proposition, if $X_1(t)=\ds\sum_{i\geq m} a_it^i$, with $a_m\neq 0$, then $m$ is called the \emph{multiplicity} of $X$ along $S$ in $p$, and denoted as $i_p(X,S)$. 

\begin{proposition}\label{multiexplo} 
Let $\F$ be a foliation in $\P^2$ given, in coordinates $(x,y,\C^2)$ by the polynomial form $\omega=P dy-Q dx$. Let $p$ be a singularity of $\F|_{\C^2}$ and $B$ a local branch of $\F$ passing through $p$, and let $\pi$ be a blow-up on $p$.
Denoting $\tilde{\F}=\pi^*\F$, $B'=\pi^*B$, $E=\pi^{-1}(p)$ and $p'\in D\cap B'$. Then
$$
i_p(\F,B) =i_{p'}(\pi^*\F,B')+m_p(B)(\nu_p(\F)-1),\quad \mbox{ if $\pi$ es non-dicritical}\\
$$
$$
i_p(\F,B)=i_{p'}(\pi^*\F,B')+m_p(B)\nu_p(\F),
\quad\quad\quad\:\: \mbox{if $\pi$ is dicritical.}
$$
\end{proposition}
Let $C$ be an irreducible curve on $\P^2$ of degree $m$ and let $\F$ be a foliation of degree $n$ having $C$ as a separatrix. For each singularity $p$ of $\F$ such that $p\in C$, and each local branch $B$ of $C$ passing through $p$ .

To calculate the genus of an irreducible algebraic curve we recall the following theorem due to Cerveau and Lins-Neto.
\begin{theorem}[Cerveau and Lins-Neto {\cite{CeLi}}]\label{cerveaulins} 
Let $\F$ be a foliation of degree $d$ in $\P^2$, and let $C$ be an irreducible curve on $\P^2$ of degree $m$.
If $C$ is a separatrix of $\F$ then

\begin{equation}\label{eqcl}
 \mathcal{X}(C)+m(d-1)=\sum_{p\in C}\sum_{B\in C\{p\}}i_{p}(\F_1,B),
\end{equation}
where $\mathcal{X}(C)$ is the Euler characteristic of the normalized curve of $C$, and $C\{p\}$ 
is the set of local branches of $C$ passing through $p$.
\end{theorem}

%
%
%
%

\section{Classification of reversible foliations}\label{sec:reversible}
In this section, we study foliations which have first integrals of the form
\begin{equation}\label{eq:eqrever}
f(x,y,z)=\dfrac{x^{p}(y^2+ax^2+bxz+cz^2)^q}{z^{p+2q}}, 
\end{equation} 
where $p\in\Z\setminus\{0\}$, $q\in\N$ and $a, b, c\in \R$, with $b^2-4ac\neq 0$. Moreover from now on, we will assume that $\gcd(p,q)=1$. Note that this implies that $p+2q\neq 0$.

Note that $df$ induces a foliation $\F$ on $\P^2$ given by the 1-form
\begin{multline}\label{eq:omegarev}
\omega = z(a(p+2q)x^2+py^2+b(p+q)xz+cpz^2)dx\\+2qxyzdy-x(a(p+2q)x^2+(p+2q)y^2+b(p+q)xz+cpz^2)dz
\end{multline}
By straightforward calculations, we obtain
\[
\sing(\F)=
\left\{
\begin{array}{c}
P_1=[0:1:0],\quad P_2=[0:i\sqrt{c}:1],\quad P_3=[0:-i\sqrt{c}:1]\\
P_4=[b(p+q)+\sqrt{\Delta}:0:-2a(p+2q)],\\
P_5=[b(p+q)- \sqrt{\Delta}:0:-2a(p+2q)],\\
P_6=[1:i\sqrt{a}:0],\qquad P_7=[1:-i\sqrt{a}:0]
\end{array}
\right\},
\]
where $\Delta=b^2(p+q)^2-4acp(p+2q)$.

Our analysis will depend on the values of $p$, $p+2q$, $p+q$, among others. In order to simplify the cases that we are going to study,  we will first reduce certain cases to others.
For instance, note that applying the automorphism $[x:y:z]\mapsto[z:y:x]$ on~\eqref{eq:eqrever}, it is enough to consider the case $p<0$.  We now divide our analysis in two cases: $p+2q>0$ and $p+2q<0$.

\subsection{Case $p+2q>0$}

In this case the first integral takes the form
\begin{equation}\label{eq:eqrever2}
f(x,y,z)=\dfrac{(y^2+ax^2+bxz+cz^2)^q}{z^{p+2q}x^{-p}}, 
\end{equation} 
And the generic fiber $C$ is
\[
C\::\:(y^2+ax^2+bxz+cz^2)^q-z^{p+2q}x^{-p}=0.
\]

Lets assume first that $a\neq 0$ and $c\neq 0$.
\begin{proposition}\label{pro:rever:acneq0}
Assume $a\neq 0$ and $c\neq 0$ and let 
\[
C\::\:(y^2+ax^2+bxz+cz^2)^q-z^{p+2q}x^{-p}=0.
\]
be the generic fiber of~\eqref{eq:eqrever2}. Then $C$ is an elliptic curve if, and only if, 
\[
(p,q)\in\{(-1,2),(-3,2)\}.
\]
\end{proposition}
\begin{proof}
Let $\F$ be the foliation induced by $df$. In this case
\[
S_C=\sing(\F)\cap C=\{P_2,P_3,P_6,P_7\}.
\]	

Moreover, $\deg(C)=2q$ so, by Theorem~\ref{cerveaulins},
\begin{equation}\label{eqgen1}
2-2\gen(C)=\sum_{P\in {S_C}}i(\F,B_{P})-2q(2-1),
\end{equation}
where $B_{P}$ are the local branches of $C$ at $P$.  

Let us calculate $i(\F,B_{P_2})$, the remaining multiplicities are analogous. 
Locally, in $z=1$, we can write $P_2=(0,i\sqrt{c})$, $C$ is given by
\[
(y^2+ax^2+bx+c)^q=x^{-p}
\]
and $\F$ is locally defined in $P_2$ by  
\[
\omega = (2ip\sqrt{c}y+b(p+q)x+a(p+2q)x^2)dx+(2qi\sqrt{c}x+2qxy)dy
\]
Therefore the eigenvalues associated to $P_2$ are $2q\sqrt{c}i$ and $-2p\sqrt{c}i$.  Hence we have two possibilities:
\begin{enumerate}
	\item If $-\dfrac{q}{p} \notin \N\cup \dfrac{1}{\N}$, by Poincar\'e's normal form theorem, there exists a biholomorphism $\varphi:(U,0) \to (V,0)$, $(x,y) \mapsto (u,v)$, with $\varphi(0)=0$, such that 
	$\varphi^*(\omega)=pvdu+qudv$.
	\item 
	If $-\dfrac{q}{p} \in \N\cup \dfrac{1}{\N}$, by Dulac's normal form theorem, 
	there exists a biholomorphism $\varphi:(U,0) \to (V,0)$, $(x,y) \mapsto (u,v)$, with $\varphi(0)=0$, such that  $\varphi^*(\omega)=vdu+(\lambda u+\eps v^{\lambda})dv$, where $\lambda$ is the natural number between $-\dfrac{p}{q}$ or $-\dfrac{q}{p}$, and $\eps \in \{0,1\}$.
	Moreover $\varepsilon=0$, since $\F$ has a first integral.
\end{enumerate}
In both cases, without loss of generality, we may assume that 
\[
\omega=pvdu+qudv \quad \text{and} \quad C=\{v^q-u^{-p}=0\}.
\]
Since $(p,q)=1$, $C$ is the only branch passing through $P_2$, hence $B_{P_2}=C$ and $i(\F,B_{P_2})=1$.

Using the same technique, we can prove that there is only one branch of $C$ passing through $P_3$, $P_6$ and $P_7$, such that $i(\F,B_{P_i})=1$. Replacing these values in~\eqref{eqgen1}, we obtain
\[
\gen(C)=q-1.
\]
Therefore, $C$ is an elliptic curve if, and only if, $q=2$. Since $p+2q>0$, $p<0$ and $\gcd(p,q)=1$, we conclude that $p=-1,-3$.
\end{proof}

We now assume that $a\neq 0$ and $c=0$.

\begin{proposition}\label{pro:rever:aneq0c0}
Assume $a\neq 0$ and $c=0$ and let 
\[
C\::\:(y^2+ax^2+bxz)^q-z^{p+2q}x^{-p}=0.
\]
be the generic fiber of~\eqref{eq:eqrever2}. Then $C$ is an elliptic curve if, and only if, 
\[
(p,q)\in\{(-1,2),(-2,3),(-4,3),(-5,3),(-5,4),(-7,4)\}.
\]

\end{proposition}

\begin{proof}
Let $\F$ be the foliation induced by $df$. In this case
\[
S_C=\sing(\F)\cap C=\{P_2,P_6,P_7\},
\]	
where $P_2=P_3=[0:0:1]$, since $c=0$. Moreover, $b^2-4ac\neq 0$ implies $b\neq 0$.
	
We now repeat the steps in the proof of Proposition~\ref{pro:rever:acneq0}. Note that  $\deg(C)=2q$ so, by Theorem~\ref{cerveaulins},
\begin{equation}\label{eq:rever:aneq0c0}
2-2\gen(C)=\sum_{P\in {S_C}}i(\F,B_{P})-2q,
\end{equation}
where $B_{P}$ are the local branches of $C$ at $P$. In the same way, we can prove that $i(\F,B_{P_6})=i(\F,B_{P_7})=1$. Remains to calculate the multiplicity of $P_2$.
	
Let us calculate $i(\F,B_{P_2})$.
In $U=\{z=1\}$, we can write $P_2=(0,0)$, $C$ is given by
\begin{equation}\label{eq:rever:genfib1}
(y^2+ax^2+bx)^q-x^{-p}=0
\end{equation}
and $\F$ is locally defined in $P_2$ by  	
\[
\omega = (a(p+2q)x^2+py^2+b(p+q)x)dx+2qxydy.
\]

Note that $P_2$ is a nilpotent singularity so, to calculate the branches passing through $P_2$, we need to do blow-ups.
Let $\pi_1:\tilde{U}\to U$ the blow-up in $P_2$, then the induced foliation  $\F'=\pi_1^*\F$ is given, in coordinates $(u,y)$, by
\begin{multline*}	
\omega'=\pi^*_1\omega=u(a(p+2q)u^2y+(p+2q)y+b(p+q)u)dy\\
+y(a(p+2q)u^2y+py+b(p+q)u)du. 
\end{multline*}

The strict transformation $C'=\pi^*_1C$ of $C$ will depend on the sign of $p+q$. We first assume that $p+q\geq 0$. In this case, $C'$ is given, in coordinates $(u,y)$, by $C'\::\: y^{p+q}(y+au^2y+bu)^q-u^{-p}=0$. Moreover, $\sing\left(\F'\right)\cap C'=\{p'=(u,y)=(0,0)\}$.  

Blowing-up again at $p'$, using the change of coordinates $u=ys$, $y=ru$, we obtain, in coordinates $(s,y)$,
$C''=\pi^*_2C'\::\: y^{2p+2q}(1+ay^2s^2+bs)^q-s^{-p}=0$ and
\begin{multline}	
\omega''=\pi^*_2\omega'=y(a(p+2q)y^2s^2+p+b(p+q)s)ds\\
+2s(a(p+2q)y^2s^2+(p+q)+b(p+q)s)dy\label{eq:omegablowup2}.
\end{multline}
In this case, $\sing\left(\pi^*_2\F'\right)\cap C''=\{p''=(s,y)=(0,0)\}$. Note that the eigenvalues associated to $p''$ are $-p$ and $2(p+q)$. Hence, using the same argument as in the proof of Proposition~\ref{pro:rever:acneq0},  we can assume without loss of generality, that in $p''$, $w''$ and $C''$ respectively, take the form
\[
\omega''=ypds+2(p+q)sdy, \qquad C''\::\: y^{2(p+q)}-s^{-p}=0,
\]
since $i(\F'',B_{p''})$ independs of the change of coordinates, where $B_{p''}$ is any branch of $C''$. Now, we write
\[
y^{2(p+q)}-s^{-p}=\prod_{k=0}^{m-1}(y^{2(p+q)/m}-s^{-p/m}e^{2\pi i k/m}),
\]
where $m=\gcd(2p+2q,-p)=\gcd(2q,-p)$ is the number of branches of $C''$ passing through $p''$. Hence $B_{p''}$ is locally parametrized by
\[
(t^{2(p+q)/m},t^{-p/m}e^{2\pi i k/(2p+2q)}).
\]
In particular, $i(\F'',B_{p''})=1$. Using Proposition~\ref{multiexplo}, as $p''$ is a non-dicritical singularity, we have
\[
i(\F',B_{p'})=i(\F'',B_{p''})+m_{p'}(B_{p'})(\nu_{p'}(\F')-1)=1+m_{p'}(B_{p'}).
\]
To obtain $m_{p'}(B_{p'})$, note that $e^{2\pi i k/(2p+2q)}(t^{(p+2q)/m},t^{-p/m})$ is a Puiseux parametrization of $B_{p'}$, so, by the Puiseux parametrization theorem, $m_{p'}(B_{p'})=\min\{(p+2q)/m,-p/m\}=-p/m$, where the last equality holds since $p+q\geq 0$. Altogether, we obtain
\[
i(\F',B_{p'})=1+m_{p'}(B_{p'})=1-\dfrac{p}{m}.
\]
In the same way, if $B_{p}$ is a branch of $C$, then
\[
i(\F,B_{p})=1-\dfrac{p}{m}.
\]
Replacing this information, together with the values of $i(\F,B_{P_6})$ and $i(\F,B_{P_7})$, in~\eqref{eq:rever:aneq0c0}, we have
\[
2-2\gen(C)=2+m\left(1-\dfrac{p}{m}\right)-2q.
\]
Therefore, $C$ is an elliptic curve if, and only if, 
\[ 
\gcd(2q,-p)+2=p+2q,
\] 
whose solution set is $\{(-2,3),(-1,2)\}$.

Now assume $p+q<0$.  
%
In this case %
$C'\::\:(y+au^2y+bu)^q-u^{-p}y^{-(p+q)}=0$, $\sing(\F')\cap C'=\{p'=(u,y)=(0,0)\}$ and, after a blow-up on $p'$, we obtain $i(\F'',B_{p''})=1+\dfrac{q}{m}$, with $m=\gcd(q,-2(p+q))=\gcd(q,-2p)$. Therefore $C$ is an elliptic curve if, and only if,
\[ 
\gcd(q,-2p)+2=q,
\]
whose solution set is $\{(-7,4),(-5,3),(-5,4),(-4,3)\}$.
\end{proof}
\begin{remark}
The calculation of $i(\F,B_{P_2})$ in the previous proof, independs on the value of $a$, since the linear part of~\eqref{eq:omegablowup2} does not depend on $a$.
\end{remark}

Now we consider the case $a=c=0$. Note that this implies that $b\neq 0$
\begin{proposition}\label{pro:rever:a0c0}
Assume $a=c=0$ and let 
\[
C\::\:(y^2+bxz)^q-z^{p+2q}x^{-p}=0.
\]
be the generic fiber of~\eqref{eq:eqrever2}. 
Then $C$ is an elliptic curve if, and only if, $(p,q)$ takes any of the following forms
\begin{align*}
(p,q)&=(2-6u,1+6u),&(p,q)&=(4-6u,-1+6u),\\
(p,q)&=(1-6u,2+6u),&(p,q)&=(5-6u,-2+6u),\\
(p,q)&=(-4-6u,1+6u),&(p,q)&=(-3-6u,-1+6u),\\
(p,q)&=(-5-6u,2+6u),&(p,q)&=(-1-6u,-2+6u),\\
(p,q)&=(1-2u,1+2u),&(p,q)&=(-3-2u,1+2u),
\end{align*}
for any $u\in\N$.
\end{proposition}
\begin{proof}
Let $\F$ be the foliation induced by $df$. In this case
\[
S_C=\sing(\F)\cap C=\{P_2,P_6\},
\]	
where $P_2=P_3=[0:0:1]$ and $P_6=P_7=[1:0:0]$, since $a=c=0$.  

By our previous remark, we can reuse the calculations made in the proof of Proposition~\ref{pro:rever:aneq0c0}, to obtain
\[
i(\F,B_{P_2})=
\begin{cases}
1-\dfrac{p}{m},\text{ with }m=\gcd(2q,-p)\text{ branches},&\text{if }p+q\geq 0\\
\\
1+\dfrac{q}{m},\text{ with }m=\gcd(q,-2p)\text{ branches},&\text{if }p+q<0.
\end{cases}
\]

We now repeat the steps in the proof of Proposition~\ref{pro:rever:acneq0}. Note that  $\deg(C)=2q$ so, by Theorem~\ref{cerveaulins},
\begin{equation}\label{eq:rever:a0c0}
2-2\gen(C)=\sum_{P\in {S_C}}i(\F,B_{P})-2q,
\end{equation}
where $B_{P}$ are the local branches of $C$ at $P$. In the same way, we can prove that $i(\F,B_{P_6})=i(\F,B_{P_7})=1$. Remains to calculate the multiplicity of $P_2$.
	
Let us calculate $i(\F,B_{P_2})$. 
In $U=\{x=1\}$, we can write $P_6=(0,0)$, $C$ is given by
\begin{equation}\label{eq:rever:genfib2}
(y^2+bz)^q-z^{p+2q}=0
\end{equation}
and $\F$ is locally defined in $P_6$ by  	
\[
\omega = 2qyzdy-((p+2q)y^2+b(p+q)z)dz.
\]
Observe that, doing the automorphism $[x:y:z]\mapsto[z:y:x]$ and denoting $p'=-(p+2q)$, we obtain $p'<0$, $p'+2q=-p>0$ and $p'+q=-(p+q)$, and we put ourselves again in the proof of Proposition~\ref{pro:rever:acneq0}, considering $p'$ instead of $p$. Therefore, 
\[
i(\F,B_{P_6})=
\begin{cases}
1+\dfrac{p+2q}{m},\text{ with }m=\gcd(2q,-p)\text{ branches},&\text{if }p+q< 0\\
\\
1+\dfrac{q}{m},\text{ with }m=\gcd(q,-2p)\text{ branches},&\text{if }p+q\geq 0.
\end{cases}
\]
As always, using Theorem~\ref{cerveaulins}, we obtain
\[
2-2\gen(C)+2q=
\begin{cases}
\gcd(2q,-p)-p+\gcd(q,-2p)+q,&\text{if }p+q\geq 0,\\
\gcd(q,-2p)+q+\gcd(2q,-p)+p+2q,&\text{if }p+q<0.
\end{cases}
\]

We now divide our analysis in two cases. If $p+q\geq 0$, $C$ will be an elliptic curve if, and only if, 
\[
p+q=\gcd(2q,-p)+\gcd(q,-2p).
\]
This equation gives the set of solutions
\begin{gather*}
(p,q)=(2-6u,1+6u),\quad (p,q)=(4-6u,-1+6u),\\
(p,q)=(1-6u,2+6u),\quad (p,q)=(5-6u,-2+6u),\\
(p,q)=(1-2u,1+2u),
\end{gather*}
for every $u\in\N$.

On the other hand, if $p+q<0$, $C$ will be an elliptic curve if, and only if, 
\[
p+q+\gcd(2q,-p)+\gcd(q,-2p)=0,
\]
whose set of solutions is 
\begin{gather*}
(p,q)=(-4-6u,1+6u),\quad (p,q)=(-3-6u,-1+6u),\\
(p,q)=(-5-6u,2+6u),\quad (p,q)=(-1-6u,-2+6u),\\
(p,q)=(-3-2u,1+2u),
\end{gather*}
for every $u\in\N$.
\end{proof}

\subsection{Case $p+2q<0$}

In this case the first integral takes the form
\begin{equation}\label{eq:eqrever3}
f(x,y,z)=\dfrac{z^{-(p+2q)}(y^2+ax^2+bxz+cz^2)^q}{x^{-p}}, 
\end{equation} 
And the generic fiber $C$ is
\[
C\::\:z^{-(p+2q)}(y^2+ax^2+bxz+cz^2)^q-x^{-p}=0,
\]
whose degree is $\deg(C)=-p$.

We now follow the steps of the previous case, obtaining the following table:
\[
{\def\arraystretch{1.9}
\begin{array}{|c|c|c|c||c|c|}
\cline{2-6}
\multicolumn{1}{c}{}    &\multicolumn{3}{|c||}{c\neq 0}&\multicolumn{2}{c|}{c=0}\\
\hline 
\sing(\F)\cap C&P_1&P_2&P_3&P_1&P_2=P_3\\
\hline
i(\F,B_P)& (-p,2q)&1&1&(-p,2q)&(q,-2p)+q\\
\hline
2-2\gen(C)&\multicolumn{3}{|c||}{p+2+\gcd(-p,2q)}&\multicolumn{2}{c|}{p+q+\gcd(-p,2q)+\gcd(q,-2p)}\\
\hline
\end{array}}
\]
In conclusion, we have the following proposition.
\begin{proposition}\label{pro:rever:case32}
Let $C$ be the generic fiber of~\eqref{eq:eqrever3}. If $c\neq 0$ then $C$ is an elliptic curve if, and only if, 
\[
(p,q)\in\{(-4,1),(-3,1)\}.
\]
On the other hand, if $c=0$ then $C$ is an elliptic curve if, and only if, 
\[
(p,q)\in\{(-4,1),(-3,1),(-5,2)\}.
\]
\end{proposition}

Combining Propositions~\ref{pro:rever:acneq0}, \ref{pro:rever:aneq0c0}, \ref{pro:rever:a0c0} and \ref{pro:rever:case32}, we obtain the following theorem.

\begin{theorem}\label{teo:reversible}
Let $f$ as in~\eqref{eq:eqrever} and let $\F$ be the foliation induced by $df$. Then $\F$ is elliptic if, and only if, after an automorphism of $\P^2$, it has a first integral of the form:

\begin{enumerate}
 \item If $p+2q>0$, $a\neq 0$ and $c\neq 0$:
 \begin{gather*}
f(x,y,z)=\dfrac{(y^2+ax^2+bxz+cz^2)^2}{xz^{3}},\quad f(x,y,z)=\dfrac{(y^2+ax^2+bxz+cz^2)^2}{x^3z}\\
\end{gather*}
\item If $p+2q>0$, $ab\neq 0$ and $c=0$:
\begin{gather*}
f(x,y,z)=\dfrac{(y^2+ax^2+bxz)^2}{xz^{3}},\quad f(x,y,z)=\dfrac{(y^2+ax^2+bxz)^3}{x^2z^{4}},\\
f(x,y,z)=\dfrac{(y^2+ax^2+bxz)^3}{x^{4}z^{2}},\quad f(x,y,z)=\dfrac{(y^2+ax^2+bxz)^3}{x^{5}z},\\
f(x,y,z)=\dfrac{(y^2+ax^2+bxz)^4}{x^{5}z^{3}},\quad f(x,y,z)=\dfrac{(y^2+ax^2+bxz)^4}{x^{7}z},
\end{gather*}
\item If $p+2q>0$, $a=c=0$:
 \begin{gather*}
f(x,y,z)=\dfrac{(y^2+bxz)^{1+6u}}{x^{-2+6u}z^{4+6u}},\quad f(x,y,z)=\dfrac{(y^2+bxz)^{-1+6u}}{x^{-4+6u}z^{2+6u}},\\
f(x,y,z)=\dfrac{(y^2+bxz)^{2+6u}}{x^{-1+6u}z^{5+6u}},\quad f(x,y,z)=\dfrac{(y^2+bxz)^{-2+6u}}{x^{-5+6u}z^{1+6u}},\\
f(x,y,z)=\dfrac{(y^2+bxz)^{1+6u}}{x^{4+6u}z^{-2+6u}},\quad f(x,y,z)=\dfrac{(y^2+bxz)^{-1+6u}}{x^{3+6u}z^{-5+6u}},\\
f(x,y,z)=\dfrac{(y^2+bxz)^{2+6u}}{x^{5+6u}z^{-1+6u}},\quad f(x,y,z)=\dfrac{(y^2+bxz)^{-2+6u}}{x^{1+6u}z^{-5+6u}},\\
f(x,y,z)=\dfrac{(y^2+bxz)^{1+2u}}{x^{-1+2u}z^{3+2u}},\quad f(x,y,z)=\dfrac{(y^2+bxz)^{1+2u}}{x^{3+2u}z^{-1+2u}},
\end{gather*}
for any $u\in\N$,
\item If $p+2q<0$ and $c\neq 0$:
 \begin{gather*}
f(x,y,z)=\dfrac{(y^2+ax^2+bxz+cz^2)z^2}{x^{4}},\quad f(x,y,z)=\dfrac{(y^2+ax^2+bxz+cz^2)z}{x^{3}},\\
\end{gather*}
\item If $p+2q<0$ and $c=0$:
 \begin{gather*}
f(x,y,z)=\dfrac{(y^2+ax^2+bxz)z^2}{x^{4}},\quad f(x,y,z)=\dfrac{(y^2+ax^2+bxz+cz^2)z}{x^{3}},\\
f(x,y,z)=\dfrac{(y^2+ax^2+bxz)^2z}{x^{5}}.\\
\end{gather*}
\end{enumerate}

\end{theorem}

\section{Classification of Lotka-Volterra foliations}\label{sec:lotkavolterra}
In this section, we study folations which have first integrals of the form
\begin{equation}\label{eq:eqLV}
f(x,y,z)=\dfrac{x^{p}y^{q}(ax+by+cz)^r}{z^{p+q+r}}, 
\end{equation} 
where $p,q\in\Z$, $r\in\N$ and $a, b, c\in \R$. Moreover from now on, we will assume that $\gcd(p,q,r)=1$.

Note that $df$ induces a foliation $\F$ on $\P^2$ given by the 1-form
\begin{multline}\label{eq:omega}
\omega =(axyz(p+r)+pyz(by+cz))dx\\+(bxyz(q+r)+qxz(ax+cz))dy\\-((p+q+r)xy(ax+by)+c(p+q)xyz)dz.
\end{multline}
Also note that if $p=0$ or $q=0$ or $p+q+r=0$, then the generic fiber of $f$ has genus zero. So, from now on, we will assume $p\neq 0$, $q\neq 0$ and $p+q+r\neq 0$. By straightforward calculations, we obtain
\[
\sing(\F)=
\left\{
\begin{matrix}
P_1=[0:0:1],&P_2=[0:1:0],&P_3=[1:0:0]\\
P_4=[0:c:-b],&P_5=[-c:0:a],&P_6=[-b:a:0]\\
\multicolumn{3}{c}{P_7=[-bcp:-acq:ab(p+q+r)]}
\end{matrix}
\right\}
\]
When $p>0$, the generic fiber of $f$ takes the form
\[
C\::\: x^py^q(ax+by+cz)^r-z^{p+q+r}=0,
\]
and $\sing(\F)\cap C=\{P_2,P_3,P_6\}$. On the other hand, if $p<0$ and $q>0$,
\[
C\::\: y^q(ax+by+cz)^r-x^{-p}z^{p+q+r}=0,
\]
and $\sing(\F)\cap C=\{P_1,P_3,P_4,P_6\}$.

%



We begin by simplifying the cases that we are going to study. First note that at most one value in $a,b,c$ can be zero. Otherwise, the generic fiber of $f$ has genus zero. Moreover, if either $a=0$ or $b=0$, we can interchange the variables $x$ and $z$, or $y$ and $z$, respectively, to obtain $ab\neq 0$. 

\begin{lemma}\label{lemlem1} In~\eqref{eq:eqLV}, if $ab \neq 0$ then it is enough to consider the following cases:
\begin{enumerate}
\item $p>0$, $q> 0$,
\item $p<0$, $q>0$.
\end{enumerate}
\end{lemma}

\begin{proof}
%
%
Assume that $p<0$ and $q<0$, then  $\dfrac{1}{f}=z^{p+q+r}x^{-p}y^{-q}(ax+by+cz)^{-r}$ is also a first integral. Hence, applying the automorphism on $\P^2$:
\[
[x:y:z] \mapsto [x:ax+by+cz:z],
\]
we obtain $p>0$ and $q<0$. Furthermore, we can further reduce this case, using  $[x:y:z] \mapsto [y:x:z]$, to the case $p<0$ and $q>0$. Note that both these transformations preserve the hypotheses $ab\neq 0$.
\end{proof}

\begin{remark}\label{rem:c0}
In~\eqref{eq:eqLV}, assume $ab\neq 0$. The case $p<0$, $q>0$ and $p+q+r<0$ is reducible to $p>0$ and $q>0$ by using the automorphism $[x:y:z] \mapsto [z:y:x]$. Note that this automorphism preserve the value of $b$, however, it preserves the fact that $a\neq 0$ if, and only if, $c\neq 0$.
\end{remark}

In view of Lemma~\ref{lemlem1} and Remark~\ref{rem:c0}, 
we have the following proposition.
\begin{proposition}\label{pro:newproLV}
Given a first integral of the form $f(x,y,z)=\dfrac{x^{p}y^{q}(ax+by+cz)^r}{z^{p+q+r}}$, we can assume, without loss of generality, that one of the following conditions hold:
\begin{enumerate}\renewcommand{\labelenumi}{\Roman{enumi}.}
 \item $ab\neq 0$, $c\in\R$ and $p>0$, $q>0$;
 \item $abc\neq 0$ and $p<0$, $q>0$, $p+q+r>0$;
 \item $a=0$, $bc\neq 0$ and $p>0$, $q>0$.
\end{enumerate}
\end{proposition}

\begin{proof}
We already observed that, without loss of generality, we may assume that $ab\neq 0$. Hence, by Lemma~\ref{lemlem1}, either $p>0$ and $q>0$ (thus having case {\it I}) or $p<0$ and $q>0$. In the latter case, we now consider, separately, the following cases:
\begin{enumerate}
 \item if $c\neq 0$ and $p+q+r>0$, we are in case {\it II};
 \item if $c\neq 0$ and $p+q+r<0$, by Remark~\ref{rem:c0}, we can reduce this case back to case {\it I};
 \item if $c=0$ and $p+q+r<0$, again, by Remark~\ref{rem:c0}, we can reduce this case to case {\it III};
 \item and finally, if $c=0$ and $p+q+r>0$, we consider the birrational map 
 \[
 [x:y:z]\mapsto [xz:yx:z^2],
 \] 
thus obtaining
 \begin{equation}\label{eq:mapbir}
 C\::\: x^py^q(ax+by)^r-z^{p+q+r},\qquad C'\::\: x^{p+q+r}y^q(by+az)^r-z^{p+2q+2r}.
 \end{equation}
Note that $C$ and $C'$ have the same genus, so we obtain case {\it III}.\qedhere
\end{enumerate}
\end{proof}

\begin{remark}\label{rem:LV}
Note that equation~\eqref{eq:mapbir} allows us to obtain first integrals which meet $c=0$, $p<0$, $q>0$ and $p+q+r>0$ (item 4, in the proof above) from first integrals obtained after analyzing case III.  More precisely, if we obtain $f(x,y,z)=\dfrac{x^{p}y^{q}(by+cz)^{r}}{z^{p+q+r}}$ as a first integral associated to an elliptic foliation in case III, with $p>0$ and $q>0$, then
\[
f(x,y,z)=\dfrac{x^{p-q-r}y^{q}(cx+by)^{r}}{z^{p}}
\]
is also a first integral associated to an elliptic foliation, which will satisfy the conditions of item 4, above, whenever $p<q+r$.
\end{remark}


%
%
In view of the previous proposition, it is enough to study, separately, the following cases:
\begin{description}
\item[Case I: $ab\neq 0$, $c\in\R$ and $p>0$, $q>0$:] Section~\ref{sec:LV:an0cn0}. 
\item[Case II: $abc\neq 0$ and $p<0$, $q>0$, $p+q+r>0$:] Section~\ref{sec:abcneq0}. 
\item[Case III: $a=0$, $bc\neq 0$ and $p>0$, $q>0$:] Section~\ref{sec:LV:an0c=0}
\end{description}

\subsection{Case I: $ab\neq 0$, $c\in\R$ and $p>0$, $q>0$}\label{sec:LV:an0cn0} 



\begin{proposition}\label{prop8.5}
Assume $p>0$ and $q>0$ and let
\[
C\::\: x^py^q(ax+by+cz)^r-z^{p+q+r}=0
\]
be the generic fiber of~\eqref{eq:eqLV}.
Then $C$ is an elliptic curve if, and only if 
\begin{equation}\label{eq:LV:caseI}
p+q+r=\gcd(q,p+r)+\gcd(p,q+r)+\gcd(r,p+q). 
\end{equation}
\end{proposition}
\begin{proof}
Let $\F$ be the foliation induced by $df$, where $f$ is as in~\eqref{eq:eqLV}.
Since $p>0$ and $abc\neq 0$, 
\[
S_C=\sing(\F)\cap C=\{P_2,P_3,P_6\}, 
\]
where $P_2=[0:1:0]$, $P_3=[1:0:0]$ and $P_6=[-b:a:0]$. Moreover, $\deg(C)=p+q+r$ so, by Theorem~\ref{cerveaulins},
\begin{equation}\label{eqgen1lv}
 2-2\gen(C)=\sum_{P\in {S_C}}i(\F,B_{P})-(p+q+r)(2-1),
\end{equation}
where $B_{P}$ are the local branches of $C$ at $P$.

We now calculate $\ds\sum i(\F,B_{P_3})$, following the arguments in the proof of Proposition~\ref{pro:rever:acneq0}. In this case, the eigenvalues associated to $P_3$ are $aq$ and $a(p+q+r)$, so there exists $m=\gcd(q,p+q+r)=\gcd(q,p+r)$ local branches $B_j$ of $C$ passing through $P_3$ such that $i(\F,B_j)=1$, for each $j=1,\ldots, m$.

Analogously, the eigenvalues associated to $P_2$ are $bp$ and $b(p+q+r)$ so there exist $n=\gcd(p,p+q+r)=\gcd(p,q+r)$ local branches $B_j^{'}$ of $P_2$ such that $i(\F,B^{'}_j)=1$; and in the same way, the eigenvalues associated to $P_6$ are $b(p+q+r)$ and $br$, and there exists $l=\gcd(r,p+q+r)=\gcd(r,p+q)$ local branches $B^{''}_j$ of $P_6$ such that $i(\F,B^{''}_{j})=1$. Hence 
\[
\sum_{P\in S_C}i(\F,B_{P})=m+n+l,
\]
and, replacing the last equality in~\eqref{eqgen1lv}, we obtain
\begin{equation}\label{eqgen2}
 2-2\gen(C)=m+n+l-(p+q+r).
\end{equation}
Therefore $C$ is an elliptic curve if, and only if,
\[
p+q+r=m+n+l,
\]
which is precisely~\eqref{eq:LV:caseI}.
\end{proof}

Now we solve equation~\eqref{eq:LV:caseI}. For the sake of clarity, the proof of the following lemma, and also the proofs of similar lemmas in the subsequent sections, will be in Appendix~\ref{app:lemma}.

\begin{lemma}\label{lem:caseI}
The 3-tuple $(p,q,r)\in\N^3$, with $\gcd(p,q,r)=1$, is a solution of equation~\eqref{eq:LV:caseI} if, and only if,
\[
(p,q,r)\in\{(1,1,1),(1,1,2),(1,2,3)\}.
\]
\end{lemma}

\subsection{Case II: $abc\neq 0$ and $p<0$, $q>0$, $p+q+r>0$}\label{sec:abcneq0}

\begin{proposition}\label{prop8.6}
Assume $p<0$, $q>0$ and $p+q+r>0$, and let 
\[
C\::\:y^q(ax+by+cz)^r-x^{-p}z^{p+q+r}=0
\]
be the generic fiber of~\eqref{eq:eqLV}. 
Then $C$ is an elliptic curve if, and only if, 
\begin{equation}\label{eq:main}
q+r=\gcd(-p,q)+\gcd(-p,r)+\gcd(q,p+q+r)+\gcd(r,p+q+r).
\end{equation}
\end{proposition}
\begin{proof}
The proof is analogous to the proof of Proposition~\ref{prop8.5}. Let $\F$ be the foliation induced by $df$, where $f$ is  as in~\eqref{eq:eqLV}. Since $p<0$ and $abc\neq 0$,
\[
S_C=\sing(\F)\cap C=\{P_1,P_3,P_4,P_6\},
\]
where $P_1=[0:0:1]$, $P_3=[1:0:0]$, $P_4=[0:c:-b]$ and $P_6=[-b:a:0]$. Also, $\deg(C)=q+r$ so, by Theorem~\ref{cerveaulins},
\[
2-2\gen(C)=\sum_{P\in S_C}i(\F,B_P)-(q+r)(2-1),
\]
where $B_P$ are the local branches of $C$ at $P$. Calculating $i(\F,B_P)$ in the same way as in the proof of Proposition~\ref{prop8.5}, we obtain that the eigenvalues associated to $P_1$, $P_3$, $P_4$ and $P_6$ are, respectively, $\{cq,-cp\}$, $\{aq,a(p+q+r)\}$, $\{br,-bp\}$ and $\{ar,a(p+q+r)\}$. Moreover, for $P_1$ (respectively for $P_3$, $P_4$ and $P_6$) there exists $m=\gcd(-p,q)$ local branches with multiplicity one (respectively, $k=\gcd(q,p+q+r)$, $n=\gcd(-p,r)$ and $l=\gcd(r,p+q+r)$).

Thus, replacing these numbers in~\eqref{eqgen1lv}, we have
\[
 2-2\gen(C)=m+n+l+k-(q+r).
\]
Therefore $C$ is an elliptic curve if, and only if,
\[
q+r=m+n+l+k,
\]
that is, equation~\eqref{eq:main}.
\end{proof}

%
%

%

The following Lemma contains the solutions of equation~\eqref{eq:main}. As before, its proof is available in Appendix~\ref{app:lemma}.
\begin{lemma}\label{lem:main-new}
The 3-tuple $(-p,q,r)\in\N^3$ is a solution of equation~\eqref{eq:main} if, and only if, it belongs to the following table
\[
\begin{array}{ccc||ccc||ccc||ccc}
-p&q&r&-p&q&r&-p&q&r&-p&q&r\\
\hline\hline
2&1&3&2&3&1&1&2&2&3&2&2\\
2&1&4&2&4&1&1&2&3&1&3&2\\
3&1&4&3&4&1&4&2&3&4&3&2\\
3&1&6&3&6&1&1&3&4&1&4&3\\
4&1&6&4&6&1&6&3&4&6&4&3
\end{array}
\]
provided that $(-p,q,r)=1$.
\end{lemma}

Observe that the automorphism $[x:y:z]\mapsto [x:ax+by+cz:z]$ allows us to interchange the values of $q$ and $r$ in~\eqref{eq:eqLV}. In the same way, the automorphism $[x:y:z]\mapsto [z:y:x]$ allows us to interchange the values of $-p$ and $p+q+r$. Therefore, from the above table we will only consider the solutions
\[
(-p,q,r)\in \{(2,1,3),(1,2,2),(2,1,4),(1,2,3),(3,1,6),(1,3,4)\}
\]

We must remark that in~\cite[p. 3554]{Gau1}, due to a small overlook while resolving equation~\eqref{eq:main}, the author could not find the full solution set showed in the table in Lemma~\ref{lem:main-new}.

\subsection{Case III: $a=0$, $bc \neq 0$ and $p>0$, $q>0$}\label{sec:LV:an0c=0}
In this case, we have a first integral of the form
\begin{equation}\label{eq:eqLVa0}
f(x,y,z)=\dfrac{x^{p}y^{q}(by+cz)^r}{z^{p+q+r}}.
\end{equation} 



\begin{proposition}\label{propvolterra2}
Assume $p>0$ and $q>0$, and let 
\[
C\::\:x^py^q(by+cz)^r-z^{p+q+r}=0
\]
be the generic fiber of~\eqref{eq:eqLVa0}. Then $C$ is an elliptic curve if, and only if, 
\begin{equation}\label{eq:LV:caseIII}
p=\gcd(p,r)+\gcd(p,q)+\gcd(p,q+r).
\end{equation}
\end{proposition}
\begin{proof}
Let $\F$ be the foliation induced by $df$, where $f$ is as in~\eqref{eq:eqLVa0}, that is
\begin{multline}\label{eq:omega:a0}
\omega =pyz(by+cz)dx+(bxyz(q+r)+qcxz^2)dy\\-((p+q+r)bxy^2+c(p+q)xyz)dz.
\end{multline}
Besides, it is straightforward to verify that
\[
S_C=\sing(\F)\cap C=\{P_2,P_3\}
\]
where $P_2=[0:1:0]$ and $P_3=[1:0:0]$. 
Also, $\deg(C)=p+q+r$ so, by Theorem~\ref{cerveaulins},
\begin{equation}\label{eqgen11}
2-2\gen(C)=\sum_{P\in S_C}i(\F,B_P)-(p+q+r)(2-1),
\end{equation}
where $B_P$ are the local branches of $C$ at $P$. 

It is a straightforward calculation to verify that there exist $l=\gcd(p,q+r)$ local branches of $C$ at $P_2$ with multiplicity one, therefore $\ds\sum_{B_{P_2}}i(\F,B_{P_2})=l$. So it remains to analyze $i(\F,B_{P_3})$. 
Locally, in $U=\{x=1\}$, we can write $P_3=(0,0)$, $C$ is given by
\[
y^q(by+cz)^r-z^{p+q+r}=0
\]
and
\[
\omega =(byz(q+r)+qcz^2)dy-((p+q+r)by^2+c(p+q)yz)dz. 
\]
Note that $P_3$, the associated linear part of $\omega$ is null. Let $\pi_{P_3}:\tilde{U}\to U$ the blow-up in $P_3$, and let $\pi^*_{P_3}C$ be the strict transformation of $C$ and $\pi_{P_3}^*\F$ be the induced foliation.

Then
$\sing\left(\pi^*_{P_3}\F\right)\cap\pi^*_{P_3}C=\{p_1,p_2\}$, where  $p_1=(0,0)$ and $p_2=(0,-c/b)$ are in coordinates $(z,t)$. Moreover, in such coordinates
\begin{gather}
\pi^*_{P_3}C\::\: t^q(bt+c)^r-z^p=0,\nonumber\\
\pi^*_{P_3}\omega=-pt(bt+c)dz+z(bt(q+r)+qc)dt\label{eq:omegablowuplotka}.
\end{gather}
Therefore, there exist $m=\gcd(p,q)$ branches of $\pi^*_{P_3}(C)$ in $p_1$ such that $i(\pi^*\F,\pi^*B)=1$ so, using Proposition~\ref{multiexplo}, there exist $m$ branches of $C$ in $P_3$ associated to $p_1$ such that $i(\F,B)=1+\dfrac{q}{m}$. 
$i(\F,B_{P_2})=1+\dfrac{r}{n}$. Replacing in~\eqref{eqgen11},
\[
 2-2\gen(C)=\sum_{i=0}^m(1+\frac{q}{m})+\sum_{i=0}^n(1+\frac{r}{n})+
\sum_{i=0}^l-(p+q+r)(2-1).
\]
Thus, $\gen(C)=1$ if, and only if, $p=m+n+l$.

\end{proof}

\begin{lemma}\label{lem:caseIII}
Consider the following table.
\[
\begin{array}{ccc||ccc||ccc||ccc}
p_0&q_0&r_0&p_0&q_0&r_0&p_0&q_0&r_0&p_0&q_0&r_0\\
\hline\hline
3&1&1 & 3&2&2 & 4&1&1 & 4&3&3\\ 
4&1&2 & 4&2&3 & 6&1&2 & 6&4&5\\ 
4&2&1 & 4&3&2 & 6&2&1 & 6&5&4\\ 
6&1&3 & 6&3&5 & 6&2&3 & 6&3&4\\
6&3&1 & 6&5&3 & 6&3&2 & 6&4&3
\end{array}
\]
The $3$-tuple $(p,q,r)\in\N^3$, with $\gcd(p,q,r)=1$, is a solution of equation~\eqref{eq:LV:caseIII} if, and only if,
$p=p_0$, $q\equiv q_0\mod p_0$ and $r\equiv r_0\mod p_0$, for a certain $(p_0,q_0,r_0)$ belonging to the given table.
Equivalently, every solution of~\eqref{eq:LV:caseIII} has the form
\[
(p,q,r)=(p_0,q_0+p_0 u,r_0+p_0 v),
\]
for some integers $u,v\geq 0$ and some $(p_0,q_0,r_0)$ in the table above.
\end{lemma}

Note that the automorphism $[x:y:z]\mapsto [x:by+cz:z]$ interchange the values of $q$ and $r$ in~\eqref{eq:LV:caseIII}. Thus, from the above table we will consider only the solutions
\begin{multline*}
(p_0,q_0,r_0)\in\{(3,1,1),(3,2,2),(4,1,1),(4,3,3),(4,1,2),\\(4,2,3),(6,1,2),(6,4,5),(6,1,3),(6,3,5),(6,2,3),(6,3,4)\}
\end{multline*}

We now combine Propositions~\ref{prop8.5}, \ref{prop8.6}, \ref{propvolterra2}, along with Lemmas~\ref{lem:caseI}, \ref{lem:main-new}, \ref{lem:caseIII} and Remark~\ref{rem:LV} to obtain the following theorem.
\begin{theorem}\label{teo:lotkavolterra}
Let $f$ as in~\eqref{eq:eqLV} and let $\F$ be the foliation induced by $df$. Then $\F$ is elliptic if, and only if, after an automorphism of $\P^2$, it has a first integral of the form:
\begin{enumerate}\renewcommand{\theenumi}{\Roman{enumi}}
 \item $ab\neq 0$, $c\in\R$ and $p>0$, $q>0$:
\begin{gather*}
f(x,y,z)=\dfrac{xy(ax+by+cz)}{z^3}, \quad f(x,y,z)=\dfrac{xy(ax+by+cz)^2}{z^4}, \\
f(x,y,z)=\dfrac{xy^2(ax+by+cz)^3}{z^6},
\end{gather*}
\item $abc\neq 0$ and $p<0$, $q>0$, $p+q+r>0$:
\begin{gather*}
f(x,y,z)=\dfrac{y(ax+by+cz)^3}{x^2z^2},\quad f(x,y,z)=\dfrac{y^2(ax+by+cz)^2}{xz^{3}},\\
f(x,y,z)=\dfrac{y(ax+by+cz)^4}{x^2z^{3}},\quad f(x,y,z)=\dfrac{y^2(ax+by+cz)^3}{xz^{4}},\\
f(x,y,z)=\dfrac{y(ax+by+cz)^6}{x^3z^{4}},\quad f(x,y,z)=\dfrac{y^3(ax+by+cz)^4}{xz^{6}},
\end{gather*}
\item $a=0$, $bc\neq 0$ and $p>0$, $q>0$:
\begin{gather*}
f(x,y,z)=\dfrac{x^{3}y^{1+3u}(by+cz)^{1+3v}}{z^{5+3(u+v)}},\quad f(x,y,z)=\dfrac{x^{3}y^{2+3u}(by+cz)^{2+3v}}{z^{7+3(u+v)}},\\
f(x,y,z)=\dfrac{x^{4}y^{1+4u}(by+cz)^{1+4v}}{z^{6+4(u+v)}},\quad f(x,y,z)=\dfrac{x^{4}y^{3+4u}(by+cz)^{3+4v}}{z^{10+4(u+v)}},\\
f(x,y,z)=\dfrac{x^{4}y^{1+4u}(by+cz)^{2+4v}}{z^{7+4(u+v)}},\quad f(x,y,z)=\dfrac{x^{4}y^{2+4u}(by+cz)^{3+4v}}{z^{9+4(u+v)}},\\
f(x,y,z)=\dfrac{x^{6}y^{1+6u}(by+cz)^{2+6v}}{z^{9+6(u+v)}},\quad f(x,y,z)=\dfrac{x^{6}y^{4+6u}(by+cz)^{5+6v}}{z^{15+6(u+v)}},\\
f(x,y,z)=\dfrac{x^{6}y^{1+6u}(by+cz)^{3+6v}}{z^{10+6(u+v)}},\quad f(x,y,z)=\dfrac{x^{6}y^{3+6u}(by+cz)^{5+6v}}{z^{14+6(u+v)}},\\
f(x,y,z)=\dfrac{x^{6}y^{2+6u}(by+cz)^{3+6v}}{z^{11+6(u+v)}},\quad f(x,y,z)=\dfrac{x^{6}y^{3+6u}(by+cz)^{4+6v}}{z^{13+6(u+v)}},
\end{gather*}
for every $u,v\geq 0$, integers.
\end{enumerate}
In addition, for every first integral of the form
\[
f(x,y,z)=\dfrac{x^py^q(by+cz)^r}{z^{p+q+r}}
\]
in case III, with $p<q+r$, we must also consider a first integral of the form
\[
f(x,y,z)=\dfrac{x^{p-q-r}y^q(ax+by)^r}{z^{p}}=\dfrac{y^q(ax+by)^r}{x^{-p'}z^{p'+q+r}},\qquad p'=p-q-r.
\]
\end{theorem}


\section{Linear families of foliations}\label{sec:Linear}
Let $\F$ and $\G$ two distinct foliations on $X$ with isolated singularities, such that $N_{\F}=N_{\G}$. Then, there exist an open covering $\U=\{U_i\}_{i\in I}$ of $X$ and families $(\omega_i)_{i\in I}$, $(\eta_i)_{i\in I}$,
 $(g_{ij})_{U_{ij} \neq \emptyset}$, such that:
\begin{enumerate}
\item $\omega_i$ and $\eta_i$ are holomorphic 1-forms on $U_i$, which define $\F$ and $\G$ in $U_i$, respectively;
\item if $U_{ij}\neq\emptyset$, then $\omega_i= g_{ij}\omega_j$ and $\eta_i= g_{ij}\eta_j$ in $U_{ij}$, where $U_{ij}=U_i \cap U_j$.
\end{enumerate}
Condition 2 implies that the 3-tuple $\{U_i,\omega_i+\alpha \eta_i,g_{ij}\}_{i \in I}$ define a foliation $\F_{\alpha}$, for every $\alpha\in\ovl C$.   We thus have a linear family of foliations $\{\F_{\alpha}\}_{\alpha \in \ovl{\C}}$ such that $N_{\F_{\alpha}}=N_{\F}$, for all $\alpha\in \ovl{\C}$. Note that $\F=\F_0$ and $\G=\F_{\infty}$.  From now on, we denote $\Pc(\F, \G):=\{\F_{\alpha}\}_{\alpha\in \ovl{\C}}$, which is called the \emph{pencil} generated by $\F$ and $\G$.



The \emph{tangency set} $\Delta(\Pc)$ of the pencil $\Pc=\{\F_{\alpha}\}_{\alpha\in\ovl{\C}}$ is
\[
\Delta(\Pc)=\tang(\F_0,\F_{\infty}),
\]
where $\tang(\F_0,\F_{\infty})$ is the analytic set $\tang(\F_0,\F_{\infty})\cap U_i=\{p\in X: \omega_i \wedge\eta_i(p)=0\}$.
In the same way, the \emph{singular set} of $\F_{\alpha}$, for $\alpha \in \ovl{\C}$, as the analytic set $\sing(\F_{\alpha})$ such that $\sing(\F_{\alpha})\cap U_i=\{\omega_{i,\alpha}= \omega_i+\alpha \eta_i=0\}$.

We will need the following lemma, which can be found in~\cite[Lemma~3.2.1]{LN3}.
\begin{lemma}\label{lem3cap2}
Let $\F$ and $\G$ foliations on $X$ with isolated singularities, such that $N_{\F}=N_{\G}$ and assume that $\F$ possesses an holomorphic first integral $f:X\to S$, where $S$ is a compact Riemann surface. Thus, the following hold.
\begin{enumerate}
\item If $\gen(f)=0$ then $\F=\G$.
\item If $\gen(f)=1$ and $\F\neq\G$, then $\G$ is turbulent with respect to $f$.
\item If $\gen(f)\geq 2$ and $\F\neq\G$ then $\tang(\G,F)>0$, for every regular fiber $F$ of $f$, not invariant by $\G$.
\end{enumerate}
\end{lemma}

\begin{proposition}\label{lem4cap2}
Let $\Pc=\{\F_{\alpha}\}_{\alpha\in \ovl{\C}}$ be a pencil on $X$ such that $\F_0$ have an holomorphic first integral $f:X \to S$ and every singularity of $\F_0$ is isolated. Then $\gen(f) \geq 1$.  Moreover
\begin{enumerate}
\item If $\gen(f)=1$, 
then there exist $c_1, \ldots c_k \in S$ such that $\Delta(\Pc) \subset \ds\bigcup_{j=1}^kf^{-1}(c_j)$.
\item If $\gen(f)\geq 2$ then $\Delta(\Pc)$ has a non-invariant component.
\end{enumerate}
\end{proposition}
\begin{proof}
Without loss of generality, we can assume that all singularities of $\F_{\infty}$ are also isolated and $\F_0 \neq \F_{\infty}$. Using Lemma~\ref{lem3cap2}, item {\it 1}, $\gen(f)\neq 0$, thus $\gen(f)\geq 1$.

Now take $F$ any regular fiber of $f$, then
\begin{equation}\label{eq3-1}
 T_{\F_{\infty}}\cdot F=2-2\gen(f).
\end{equation}
Indeed, if $F$ is invariant by $\F_{\infty}$ then there exists a regular fiber $F'$ of $f$ such that $F'\cap \sing(\F_0)=\emptyset$, that is, $Z(\F,F')=0$. Therefore
\[ 
T_{\F_{\infty}}\cdot F=T_{\F_0}\cdot F=T_{\F_0} \cdot F'=\chi(F')-Z(\F,F')= 2-2 \gen(f).
\] 
On the other hand, if $F$ is not invariant by $\F_{\infty}$ then
\[ 
T_{\F_{\infty}}\cdot F=F \cdot F-\tang(\F_{\infty},F)=-\tang(\F_{\infty},F)=T_{\F_0}\cdot F=2-2 \gen(f).	
\] 


Now assume that $\gen(f)=1$ and choose any component $C$ of $\Delta(\Pc)$. If $C$ is not contained on a fiber of $f$ then $C$ is a non-invariant component of $\Delta(\Pc)$. 
Given $\F_{\alpha}$ with all singularities are isolated, there exists a regular fiber $F$, not invariant by $\F_{\alpha}$, such that $F\cap C\neq\emptyset$ and, in particular, $\tang(\F_{\alpha},F)>0$. Therefore,
\[
T_{\F_{\alpha}}\cdot F=F\cdot F-\tang(\F_{\alpha},F)=T_{\F_{\infty}}\cdot F=0,
\]
which implies $\tang(\F_{\alpha},F)=0$, a contradiction. Thus, every component $C$ of $\Delta(\Pc)$ is contained on a fiber of $f$, hence, there exist $c_1, \ldots, c_k \in S$ such that
\[
 \Delta(\Pc) \subset \bigcup_{j=1}^k f^{-1}(c_j).
\]
This proves {\it 1}.

On the other hand, assume $\gen(f)\geq 2$  and let $F$ be a generic regular fiber of $f$, not invariant by $\F_{\infty}$. Hence, by~\eqref{eq3-1},
\[
\tang(\F_{\infty},F)<0.
\]
Now, if $\Delta(\Pc)$ were invariant, then there exist $c_1,\ldots,c_k\in S$ such that $\ds\Delta(\Pc) \subset \bigcup_{i=1}^k F_{c_i} $ and $F\cap\Delta(\Pc)=\emptyset$. In particular,
\[
\tang(\F_{\infty},F)= 0,
\]
a contradiction.
\end{proof}

The following example shows that, in Proposition~\ref{lem4cap2}, is necessary that every singularity of $\F_0$ must be isolated.
\begin{example}\label{ejmcontra}
Let's consider the pencil $\Pc=\Pc(\F,\G)$ on $\P^2$, where $\F$ and $\G$ are defined by the polynomial 1-forms
\begin{align*}
\omega&=-z(3 a x^2+3 y^2+ c z^2)dx+ 4 x y z dy- x(3 ax^2+3y^2 -c z^2) dz,\\
\eta&=   z^2 x dx -  x^2 z dz,
\end{align*}
where $a,c \in \C^*$ are fixed constants. Then,
\begin{enumerate}
\item given $\alpha \in \C$,
\[
H_{\alpha}=\frac{y^2+a x^2 + \alpha x z+ c z^2}{xz^3}
\]
is a first integral of $\F_{\alpha}$. In addition, $H_{\infty}=\dfrac{x}{z}$ is a first integral of $\F_{\infty}$.
\item Let $A_{\pm}=[0:\pm\sqrt{c}:1]$, $D_1=[0:1:0]$, $C_{\pm}=[1:\pm\sqrt{a}:0]$  and
\[
B_{\pm}(\alpha)=\left[\frac{-\alpha\pm\sqrt{\alpha^2+12ac}}{6a}:0:1\right].
\]
Given $\alpha\in\C$,
\[
\sing(\F_{\alpha})=\{A_{\pm},B_{\pm}(\alpha),C_{\pm},D_1\}.
\]
The singularities $A_{\pm}$, $C_{\pm}$ and $D_1$ are fixed singularities of $\F_{\alpha}$ of type $(2:1)$, $(2:3)$ and $ (1:-3)$, respectively, and $B_{\pm}(\alpha)$
of fixed type equal to $(-1:1)$. Moreover, $\sing(\F_{\infty})=\{x=0\} \cup \{z=0\}$.

\item $\Delta(\Pc)=\{x^2y z^2=0\}$, where $\{x=0\}$ and $\{z=0\}$ are 
invariant by $\Pc$, and  $\{y=0\}$  is a non-invariant.
\item For any $\alpha\in\C $ we have $\gen(F_c)=1$, where $F_c=H_{\alpha}^{-1}(c)$ and $c \in \C\setminus \{0,\frac{16}{\sqrt{3}}, \frac{-16}{\sqrt{3}}, \infty\}$.
Besides $H_{\alpha}$ possesses four critic fibers associated to $c\in \{0,\frac{16}{\sqrt{3}}, \frac{-16}{\sqrt{3}}, \infty\}$. 
\end{enumerate}
Let $q_0=C_+$, without loss of generality we can assume that there exists a local system of coordinates $(x,y,U_0)$, with $q_0\in U_0$ and $x(q_0)=y(q_0)=0$, such that $\omega=2 x dy-3 y dx$ and $\eta=x d x$ represent $\F_0$ and $\F_{\infty}$ on $U$, respectively.
Let $\pi_{q_0}:=\pi_3\circ \pi_{2}\circ \pi_1$ the desingularization process of $\F_0$ in $q_0$, then  $f=\pi_{q_0} \circ H_0$ is an elliptic fibration of $\F_0$.
For the first blow-up, there exist a system of coordinates $(x,t,U_1)$ such that,
\[
\pi_1^*(\omega+\alpha \eta)= x((\alpha-t)dx+2x dt). 
\]
Similarly, for the third blow-up there exists a system of coordinates $(s,w,U)$ such that $\pi_{q_0}^*\F_{\alpha}|_{U}$ is given by
\[ 
 s(s\omega+\alpha)d\omega+2\alpha \omega ds.
\] 
This means that there exists non-isolated singularities of $\pi_{q_0}^*\F_0$ and $\pi_{q_0}^*(\{y=0\})$ is not invariant by $\Delta(\widetilde{\Pc})$, where  $\widetilde{\Pc}=\pi_{q_0}^*(\Pc)$.


\end{example}


%
%

\begin{proposition}\label{lem4cap2-2}
Let $\Pc=\{\F_{\alpha}\}_{\alpha\in \ovl{\C}}$ be a pencil in $X$ such that $\F_0$ has an  holomorphic first integral $f:X \to S$ with isolated singularities. Then $\gen(f)=1$ if, and only if,  $\Delta(\Pc)$ is invariant.
\end{proposition}
\begin{proof}
Lemma~\ref{lem4cap2} directly implies the if part. Conversely, if $\gen(f)=1$, take $C$ any component of $\Delta(\Pc)$. Again by lemma~\ref{lem4cap2}, there exists $c\in S$ such that $C\subset f^{-1}(c)$, in particular $C$ is invariant by $\F_0$. Since $\F_0$ possesses isolated singularities and $C\subset\Delta(\Pc)$ is invariant by $\F_0$, we conclude that $C$ is invariant by $\Pc$.
\end{proof}

Example~\ref{ejmcontra} shows that in Proposition~\ref{lem4cap2-2}, the condition that $\F_0$ must have non-isolated singularities is necessary. 
In fact, every linear family induced by Gautier's classification share the property that, if we make a sequence of blow-ups to obtain an elliptic fibration, the associated linear pencil has the property that the foliation  $\pi^*(\F_0)$ have an elliptic fibration but does not have isolated singularities and possesses a non-invariant curve. Moreover, this family does not belong to any of the four types given in the following theorem, due to Lins-Neto.

\begin{theorem}[Lins-Neto~{\cite{LN3}}]\label{teo:lins7}
Let $\Pc=\{\F_{\alpha}\}_{\alpha\in \ovl{\C}}$ be a pencil in  $X$ such that $\F_0$ and $\F_{\infty}$ has all their singularities are reduced and have first integral $f:X \to S_1$  and $g:X \to S_2$, respectively.
If $\Delta(\Pc)$ is invariant then $\Pc$ is bimeromorphically equivalent to four possible types in $\P^2$:
\begin{enumerate}
\item Degree two pencil, defined by
\[
\Pc_2 \left\{\begin{array}{l}
\omega_1=(4x-9x^2+y^2)dy-(6y-12xy)dx,\\
\eta_1=(2y-4xy)dy-3(x^2-y^2)dx.
\end{array}\right.
\]
\item Degree three pencil, defined by
\[
\Pc_3\left\{\begin{array}{l}
\omega_2=(-x+2 y^2-4x^2 y+x^4)dy-y(-2-3xy+x^3)dx,\\
\eta_2=(2y-x^2+xy^2)dy-(3xy-x^3+2y^3)dx.
\end{array}\right.
\]
\item Degree four pencil, defined by
\[
\Pc_4\left\{\begin{array}{l}
\omega_3=(x^3-1)xdy-(y^3-1)ydx,\\
\eta_3=(x^3-1)y^2dy-(y^3-1)x^2dx.
\end{array}\right.
\]
\item Degree three pencil, defined by
\[
\Pc_3'\left\{\begin{array}{l}
\omega_4=(-4x+x^3+3 x y^2)dy-2 y(y^2-1)dx,\\
\eta_4=(x^2y-y^3)dy-2 x(y^2-1)dx.
\end{array}\right.
\]
\end{enumerate}
\end{theorem}

\section{Aknowledgements}
The second author was partially supported by Postdoctoral Grant CG 07-2013 FONDECYT.


\begin{appendices}
\section{Arithmetic related proofs}\label{app:lemma}
We begin with Lemma~\ref{lem:caseI} in Section~\ref{sec:LV:an0cn0}.

\begin{proof}[Proof of Lemma~\ref{lem:caseI}]
We write equation~\eqref{eq:LV:caseI} as
\[
p+q+r=m+n+l,
\]
where $m=\gcd(q,p+r)$,  $n=\gcd(p,q+r)$ and  $l=\gcd(r,p+q)$. Observe that $m\leq q$, $n\leq p$ and $l\leq r$, so we have
\[
p+q+r=m+n+l \quad \iff \quad \left\{ 
\begin{aligned}
m=\gcd(q,p+r)&=q,\\
n=\gcd(p,q+r)&=p,\\
l=\gcd(r,p+q)&=r.
\end{aligned}\right.
\]
Assume that $\gcd(q,p+r)=q$,  $\gcd(p,q+r)=p$ and  $\gcd(r,p+q)=r$, then there exist $ \alpha,\beta, \gamma \in \N$ such that
\[
p+r=\alpha q,\qquad q+r=\beta p,\qquad p+q=\gamma r.
\]
These equalities imply
\begin{align}
(\alpha+1)q&=(\beta+1)p;\label{igg1}\\
(\beta+1)p&=(\gamma+1)r;\label{igg2}\\
(\gamma\beta-1)p&=(\gamma+1)q\label{igg3},
\end{align}
and, since  $p\leq q \leq r$ , we obtain $\beta \geq
\alpha \geq \gamma$. On the other hand, combining~\eqref{igg1} y~\eqref{igg3}, we obtain
\[
\ds \frac{q}{p}=\frac{\beta+1}{\alpha+1}=\frac{\gamma\beta-1}{\gamma+1}.
\]
Thus,
\begin{equation}\label{eq4}
\alpha \gamma \beta=2+\alpha+\beta+\gamma, \quad \beta \geq
\alpha \geq \gamma
\end{equation}
Hence, we have two possibilities:
\begin{enumerate}
\item If $\alpha=\beta$, equation \eqref{eq4} implies
\[
\beta^2 \gamma=2+2\beta +\gamma \iff \gamma(\beta-1)=2.
\]
Therefore $(\beta,\gamma)=(3,1)$ or $(2,2)$, and these imply $(\alpha,\beta,\gamma)=(3,3,1)$ or $(2,2,2)$. 
Analogously, if $\alpha=\gamma$ then 
\[
\gamma^2 \beta=2+2\gamma +\beta \iff \beta(\gamma-1)=2.
\]
Together with $\gamma\leq \beta$, we obtain $\gamma=\beta=2$, which imply  $(\alpha,\beta,\gamma)=(2,2,2)$.

Thus, from equations~\eqref{igg1}, \eqref{igg2} y~\eqref{igg3}, and the values $(\alpha,\beta, \gamma)$ obtained above, we obtain $(p,q,r)=(1,1,2)$ and $(1,1,1)$.  

\item If $\beta >\alpha>\gamma$, equation~\eqref{eq4} implies $\alpha\gamma\beta<2+3\beta$ or, equivalently, $\beta(\gamma\alpha-3)<2$. So, if $\gamma\alpha>3$ then $\beta$ must be 1, a contradiction. Hence $\gamma\alpha\leq 3$ and, because of $\alpha>\gamma$, $\gamma=1$ and $\alpha=2$ or 3. However, $\alpha=3$ and $\gamma=1$ on~\eqref{eq4} imply $\beta=3$, another contradiction. 
Then $\alpha=2$ so $\beta=5$. Therefore $(\alpha,\beta,\gamma)=(2,5,1)$ and this implies $(p,q,r)=(1,2,3)$.
\qedhere 
\end{enumerate}
\end{proof}

We continue with the proof of Lemma~\ref{lem:main-new} in Section~\ref{sec:abcneq0}. We state first the following trivial  lemma.
\begin{lemma}\label{lem:1}
	Let $d=\gcd(a,b)$. If $d\neq a$ then $d\leq \dfrac{a}{2}$.
\end{lemma}
%

\begin{proof}[Proof of Lemma~\ref{lem:main-new}]
It is straightforward to verify that each 3-tuple in our table is indeed a solution, thus we have the if part. We now prove the only if part.  First, by doing the change $p=-p$, we can rewrite equation~\eqref{eq:main} as
\begin{equation}\label{eq:mainrew}
q+r=\gcd(p,q)+\gcd(p,r)+\gcd(q,q+r-p)+\gcd(r,q+r-p).
\end{equation}
Let $(p,q,r)\in\N^3$ be a solution of~\eqref{eq:mainrew} such that $\gcd(p,q,r)=1$.  From now on, we denote $m=\gcd(p,q)$, $n=\gcd(p,r)$, $k=\gcd(q,q+r-p)$ and $l=\gcd(r,q+r-p)$, so we need to solve
\begin{equation}\label{eq:lem}
q+r=m+n+l+k.
\end{equation}
	
First, let's assume that $p=q$, so we obtain $m=\gcd(q,q)=q$, $n=\gcd(q,r)$, $l=\gcd(r,r)=r$ and $k=\gcd(q,r)$. Hence
\[
q+r=m+n+l+k=q+r+2\gcd(q,r)\geq q+r+2,
\]
a contradiction. Therefore $p\neq q$. In a similar way, we can conclude that $p\neq r$.

Interchanging the values of $q$ and $r$, we can assume without loss of generality, that $q\leq r$. We divide our analysis in three main cases: $p<q\leq r$, $q<p<r$ and $q\leq r<p$.
	
Assume that $p<q\leq r$. In this case we have $r-p>0$ and $q-p>0$. Therefore $k=\gcd(q,r-p)$ and $l=\gcd(r,q-p)$. Observe that $m\leq p$, $n\leq p$, $k\leq r-p$ and $l\leq q-p$, so, if one of these inequalities is strict, adding up all of them would give
\[
q+r=m+n+l+k<q+r,
\]
a contradiction. Thus, we have $m=p=\gcd(p,q)$, $n=p=\gcd(p,r)$, $k=r-p$ and $l=q-p$. In addition, this implies that $p|q$ and $p|r$, hence $p|\gcd(p,q,r)=1$, that is, $p=1$. Therefore, $p=m=n=1$, $k=r-1=\gcd(q,r-1)$ and $l=q-1=\gcd(r,q-1)$. This in turn, implies $r-1\leq q$ and $q-1\leq r$, hence $r-1\leq q\leq r+1$.  Moreover, since $q\leq r$, we cannot have $q=r+1$, so $q=r-1$ or $q=r$. 	If $q=r-1$, then $r=q+1$ and $(q-1)|(q+1)$, that is $(q-1)|(q-1+2)$. Therefore $(q-1)|2$, so $q=2$ or $q=3$, giving us the 3-tuples $(1,2,3)$ and $(1,3,4)$, respectively. On the other hand, if $q=r$, then $(q-1)|q$. Therefore $q=2$ and we obtain the solution $(1,2,2)$. 
	
Now, we assume that $q\leq r<p$ and let $p'=q+r-p$. Then $p'<q\leq r$. Since $(p',q,r)$ is also a solution of~\eqref{eq:mainrew}, then $(p',q,r)$ must be one of the solutions obtained in the previous paragraph, that is $(p',q,r)=(1,2,3)$, $(1,3,4)$ or $(1,2,2)$. Therefore, we obtain $(p,q,r)=(4,2,3)$, $(6,3,4)$ or $(3,2,2)$, respectively.
	
Finally, let's assume that $q<p<r$. In this case we have $r-p>0$ and $p-q>0$. Therefore, $k=\gcd(q,r-p)$ and $l=\gcd(r,p-q)$. Adding up the inequalities $m\leq q$, $n\leq p$, $l\leq p-q$ and $k\leq q$, we obtain
\[
q+r=m+n+l+k\leq 2p+q,
\]
impliying $r\leq 2p$. Let's suppose that $p\neq n=\gcd(p,r)$ and $p-q\neq l=\gcd(r,p-q)$. By Lemma~\ref{lem:1}, $n\leq \dfrac{p}{2}$ and $l\leq\dfrac{p-q}{2}$. We add these inequalities, together with $m\leq q$ and $k\leq r-p$, to obtain
\[
q+r=m+n+l+k\leq r+\dfrac{q}{2},
\]
a contradiction. Therefore, $n=p$ or $l=p-q$.
	
Suppose $n=p$. Then $p=n=\gcd(p,r)$, that is $p|r$. This means that there exists $\alpha\in\N$ such that $r=\alpha p$, so  $p<r=\alpha p\leq 2p$. Therefore $\alpha=2$ and $r=2p$. Since $1=\gcd(p,q,r)=\gcd(p,q,2p)$, we obtain $m=\gcd(p,q)=1$. Furthermore, $k=\gcd(q,r-p)=\gcd(q,p)=1$ and $l=\gcd(r,p-q)=\gcd(2p,p-q)$ so equation~\eqref{eq:lem} becomes
\begin{equation}\label{eq:46}
q+p=l+2.
\end{equation}	
The  equality $l=(2p,p-q)$ implies that $l|(2q)$, so $l|(2\gcd(p,q))=2$, that is, $l=1$ or $l=2$. If $l=1$, then $q+r=m+n+l+k=3+p<3+r$. Hence $q<3$, thus $q=1$ or $q=2$. We discard $q=2$ because, from~\eqref{eq:46}, $p=l=1<q$, a contradiction since we are assuming $q<p$. Thus $q=1$ and, again from~\eqref{eq:46}, $p=2$ and $r=4$, obtaining the solution $(p,q,r)=(2,1,4)$.
On the other hand, if $l=2$, equation~\eqref{eq:46} becomes $q+p=4$. This implies, since $q<p$, that $2q<q+p=4$, so $q=1$, $p=3$ and $r=6$, obtaining $(p,q,r)=(3,1,6)$.

Now suppose $l=p-q$. Since $p-q=l=\gcd(r,p-q)$, $(p-q)|r$ and there exists $\alpha\in\N$ such that $r=(p-q)\alpha$. Besides $1=\gcd(p,q,r)=\gcd(p,q,\alpha(p-q))$, so we have $m=\gcd(p,q)=1$. Replacing the former equalities in~\eqref{eq:lem} and reordering, we obtain
\begin{equation}\label{eq:47}
(\alpha-1)p+(2-\alpha)q=1+n+k.
\end{equation}
In addition, $k=\gcd(q,r-p)=\gcd(q,(\alpha-1)p)$, $n=\gcd(p,r)=\gcd(p,\alpha q)$ and, since $\gcd(p,q)=1$, $k|(\alpha-1)$ and $n|\alpha$.   

We now study the different values $\alpha$ can take. Note that if $\alpha=1$ then $r=p-q<p$, a contradiction, so $\alpha\geq 2$. If $\alpha=2$, then $k=1$, $n=1$ or $2$ and equation~\eqref{eq:47} becomes $p=2+n$. Assume $n=1$, so $p=3$. Since $q<p=3$, $q$ can be $1$ or $2$. If $q=2$ then $r=2$ which contradicts $q<p<r$. Thus, $q=1$ and $r=4$, obtaining $(p,q,r)=(3,1,4)$. On the other hand, if $n=2$ then $p=4$. Since $q<p=4$,  $q=1,2,3$ and, respectively, $r=6,4,2$. It follows that we obtain the solution $(p,q,r)=(4,1,6)$, since $q=2$ and $3$ respectively imply $\gcd(p,q,r)=\gcd(4,2,4)=2\neq 1$ and $r<q$, both contradictions.

Now assume that $\alpha=3$. This implies that $k|2$ and $n|3$, thus $k\leq 2$, $n\leq 3$ and $k+n\leq 5$. Therefore, in~\eqref{eq:47},
\[
1\leq q<p<2p-q=1+k+n\leq 6,
\]
that is, $p$ can take the values 2, 3, 4, 5. We study these cases separately.
\begin{enumerate}
\item If $p=2$ then $q=1$ and $r=\alpha(p-q)=3$, giving the solution $(p,q,r)=(2,1,3)$.
\item If $p=3$ then $n=\gcd(p,\alpha q)=\gcd(3,3q)=3=p$. This is an already studied case.
\item If $p=4$, since $4=p<2p-q=8-q\leq 6$, we obtain $2\leq q<4$. Thus, $q$ must be $3$, because $\gcd(p,q)=1$. From this $r=\alpha(p-q)=3$, which contradicts $p<r$.
\item If $p=5$, since $5<2p-q\leq 6$ then $10-q=2p-q=6$, that is $q=4$ and $r=\alpha(p-q)=3$, which contradicts $p<r$.
\end{enumerate}

Note that at this state of the proof, we already found at least half of the solutions in our table. The remaining solutions can be obtained by interchanging the values of $q$ and $r$. Rests to prove that assuming $\alpha\geq 4$, we cannot obtain any further solutions.

Let's assume $\alpha\geq 4$. Since we already studied the case $n=p$, we may also assume that $n\neq p$, which, by Lemma~\ref{lem:1}, implies $n\leq\dfrac{p}{2}$. 
From equation~\eqref{eq:47}, we obtain the inequality
\[
3p-2q\leq (\alpha-4)(p-q)+3p-2q=1+n+k\leq 1+\dfrac{p}{2}+q,
\]
that is 
\begin{equation}\label{eq:48}
\dfrac{5}{2}p\leq 1+3q. 
\end{equation}
Rearranging again~\eqref{eq:47}, we obtain
\[
\alpha(p-q)-p+2q=1+n+k\leq 1+n+q.
\]
Thus
\[
2q<r-p+2q=\alpha(p-q)-p+2q=1+n+k\leq 1+n+q\leq 1+\dfrac{p}{2}+q,
\]
which implies $q<1+\dfrac{p}{2}$. This in turn implies
\[
5q<5+\dfrac{5}{2}p\leq 5+(1+3q)=6+3q,
\]
that is, $q<3$ or, equivalently, $q\leq 2$. Using this inequality in~\eqref{eq:48}, we obtain
$\dfrac{5}{2}p\leq 7$ or $p\leq \dfrac{14}{5}<3$. Since $1\leq q<p<3$, we conclude that $p=2$, $q=1$ and $n=1$.
Replacing in~\eqref{eq:47},
\[
2(\alpha-1)+(2-\alpha)=1+1+k\leq 2+1,
\]
that is $\alpha\leq 3$, a contradiction.
\end{proof}

Finally, we prove Lemma~\ref{lem:caseIII} in section~\ref{sec:LV:an0c=0}.
\begin{proof}[Proof of Lemma~\ref{lem:caseIII}]
We write equation~\eqref{eq:LV:caseIII} as
\[
p=m+n+l,
\]
where $m=\gcd(p,q)$,  $n=\gcd(p,r)$ and $l=\gcd(p,q+r)$. Thus, there exist  $\gamma, \alpha, \beta \in \N$ such that
\[
p=n \gamma,\quad 
p=m \beta,\quad  
p=l \alpha 
\]
Replacing these equalities in our equation, we obtain
\begin{equation}\label{eq:albetgam}
\alpha\beta\gamma=\alpha\gamma+\alpha\beta+\beta\gamma.
\end{equation}
We first take care of this equation. 

Let's assume that $\alpha=\beta$. With this, equation~\eqref{eq:albetgam} reduces to $(\alpha-2)\gamma=\alpha$, that is $(\alpha-2)|\alpha$, so $\alpha=3$ or $4$. If $\alpha=\beta=3$ then $\gamma=3$, so we obtain the solution $(\alpha,\beta,\gamma)=(3,3,3)$. On the other hand, if $\alpha=\beta=4$ then $\gamma=2$ so we obtain $(\alpha,\beta,\gamma)=(4,4,2)$. Note that, because of the symmetry of $\alpha,\beta,\gamma$ in \eqref{eq:albetgam}, we obtain the solutions $(\alpha,\beta,\gamma)=(2,4,4)$ $(4,2,4)$ if we assume $\beta=\gamma$ or $\alpha=\gamma$, respectively. 

We now assume that $\alpha$, $\beta$ and $\gamma$ are different. First, let's assume that $\alpha<\beta<\gamma$. This in~\eqref{eq:albetgam} implies $\alpha\beta\gamma<3\beta\gamma$, so $\alpha<3$. But if $\alpha=1$ then $\beta\gamma=\gamma+\beta+\beta\gamma$, a contradiction. Thus $\alpha=2$ and~\eqref{eq:albetgam} reduces to
\[
(\beta-2)\gamma=2\beta.
\]
This in turn implies that $(\beta-2)|2\beta$, hence $(\beta-2)|4$. Therefore $\beta$ can take the values $\beta=6,4,3$ implying $\gamma=3,4,6$, respectively. But we discard the case $\beta=\gamma=4$ and $\beta=6$, $\gamma=3$, since we are assuming that $\alpha<\beta<\gamma$ are different. Again, by the symmetry of equation~\eqref{eq:albetgam}, we obtain $(\alpha,\beta,\gamma)=(2,3,6),(2,6,3),(3,2,6),(6,2,3),(3,6,2),(6,3,2)$.

We now will use these solutions to obtain solutions of~\eqref{eq:LV:caseIII}. Note that, although equation~\eqref{eq:albetgam} is symmetric with respect to $\alpha$, $\beta$ and $\gamma$, we cannot freely interchange the values of $p$, $q$ and $r$ in equation~\eqref{eq:LV:caseIII}. However, that $q$ and $r$ are interchangeable, which means that we can actually interchange $\beta$ and $\gamma$. 

Now we proceed to obtain solutions of~\eqref{eq:LV:caseIII}.
\begin{enumerate}
 \item Let $(\alpha,\beta,\gamma)=(3,3,3)$. In this case $p=3m=3n=3l$, so $m=n=l$. Since $\gcd(p,q,r)=1$, we obtain $m=n=l=1$ and $p=3$. Remains to obtain values for $q$ and $r$. Observe that there exist integers $u,v$ and $t,s$ such that 
$q=3u+t$ and $r=3v+s$, with $0\leq s,t<3$. The remainders $s$ and $t$ cannot be zero, since $m=\gcd(p,q)=n=\gcd(p,r)=1$. Moreover, since $l=(p,q+r)=1$, $t+s$ cannot be a multiple of 3. Therefore, we obtain the families of solutions
\begin{gather*}
(p,q,r)=(3,3u+1,3v+1),\, u,v\geq 0,\\
(p,q,r)=(3,3u+2,3v+2),\, u,v\geq 0.
\end{gather*}
\item Let $(\alpha,\beta,\gamma)=(2,4,4)$.  In this case $p=4m=4n=2l$, so $m=n$ and $l=2m=2n$. Since $\gcd(p,q,r)=1$, we obtain $m=n=1$, $l=2$ and $p=4$. Similarly to the previous case, observe that there exist integers $u,v$ and $t,s$ such that 
$q=4u+t$ and $r=4v+s$, with $0\leq s,t<4$. Again, the remainders $s$ and $t$ cannot be zero, since $m=\gcd(p,q)=n=\gcd(p,r)=1$. Since $l=(p,q+r)=2$, $t+s$ must be even, but cannot be a multiple of 4. Therefore, we obtain the families of solutions
\begin{gather*}
(p,q,r)=(4,4u+1,4v+1),\, u,v\geq 0,\\
(p,q,r)=(4,4u+3,4v+3),\, u,v\geq 0.
\end{gather*}
\item Let $(\alpha,\beta,\gamma)=(4,2,4)$.  In this case $p=2m=4n=4l$, so $n=l$ and $m=2n=2l$. Since $\gcd(p,q,r)=1$, we obtain $n=l=1$, $m=2$ and $p=4$. Again, there exist $u,v$ and $t,s$ such that $q=4u+t$ and $r=4v+s$, with $0\leq s,t<4$. Since $m=\gcd(p,q)=2$, the only possibility for $t$ is to be $t=2$, in addition, using $l=(p,q+r)=1$, $t+s$ must be odd, thus $t=1$ or $3$. Therefore, we obtain the families of solutions
\begin{gather*}
(p,q,r)=(4,4u+2,4v+1),\, u,v\geq 0,\\
(p,q,r)=(4,4u+2,4v+3),\, u,v\geq 0.
\end{gather*}

\item Let $(\alpha,\beta,\gamma)=(2,3,6)$. In this case $p=3m=6n=2l$, so $m=2n$ and $l=3n$. 
Since $\gcd(p,q,r)=1$, $n=1$ so $m=2$, $l=3$ and $p=6$. Writing $q=6u+t$, $r=6v+s$, with $0\leq t,s<6$, and using $m=(p,q)=(6,t)=2$, $n=(p,r)=(6,s)=1$, we obtain $t=2,4$ and $s=1,5$. But this means $t+s=3,5,7,9$ together with $l=(p,q+r)=(6,t+s)=3$, imply
\begin{gather*}
(p,q,r)=(6,6u+2,6v+1),\, u,v\geq 0,\\
(p,q,r)=(6,6u+4,6v+5),\, u,v\geq 0.
\end{gather*}

\item Let $(\alpha,\beta,\gamma)=(3,2,6)$. In this case $p=2m=6n=3l$, so $l=2n$ and $m=3n$. Since $\gcd(p,q,r)=1$, $n=1$ so $l=2$, $m=3$ and $p=6$. Writing $q=6u+t$, $r=6v+s$, with $0\leq t,s<6$, and using $m=(p,q)=(6,t)=3$, $n=(p,r)=(6,s)=1$, we obtain $t=3$ and $s=1,5$. Therefore $t+s=4,8$ which, along with $l=(p,q+r)=(6,t+s)=2$ implies
\begin{gather*}
(p,q,r)=(6,6u+3,6v+1),\, u,v\geq 0,\\
(p,q,r)=(6,6u+3,6v+5),\, u,v\geq 0.
\end{gather*}

\item Let $(\alpha,\beta,\gamma)=(6,2,3)$. In this case $p=2m=3n=6l$, so $m=3l$ and $n=2l$. Since $\gcd(p,q,r)=1$, $l=1$ so $m=3$, $n=2$ and $p=6$. Writing $q=6u+t$, $r=6v+s$, with $0\leq t,s<6$, and using $m=(p,q)=(6,t)=3$, $n=(p,r)=(6,s)=2$, we obtain $t=3$ and $s=2,4$. But this means $t+s=5,7$ which, along with $l=(p,q+r)=(6,t+s)=1$ implies
\begin{gather*}
(p,q,r)=(6,6u+3,6v+2),\, u,v\geq 0,\\
(p,q,r)=(6,6u+3,6v+4),\, u,v\geq 0.
\end{gather*}

\end{enumerate}
The remaining solutions $(p,q,r)$, obtained from $(\alpha,\beta,\gamma)=(4,4,2)$, $(2,6,3)$, $(3,6,2)$, $(6,3,2)$, can be deduced directly from items 3--6 above, by interchanging $q$ and $r$.
\end{proof}

\end{appendices}

\def\cprime{$'$}

%

\end{document}